%% file: main.tex
\documentclass[letterpaper,10pt]{article}
\usepackage{geometry}
\usepackage{fullpage}

\usepackage{graphicx}%
\usepackage{multirow}%
\usepackage{amsmath,amssymb,amsfonts}%
\usepackage{amsthm}%
\usepackage{mathrsfs}%
\usepackage[title]{appendix}%
\usepackage{xcolor}%
\usepackage{textcomp}%
\usepackage{manyfoot}%
\usepackage{booktabs}%
\usepackage{algorithmic}
\usepackage{listings}%
\usepackage{caption}
\usepackage{hyperref}
\usepackage[capitalize,nameinlink,noabbrev]{cleveref}
\usepackage{mathtools}
\usepackage{cite}
\usepackage{adjustbox}
\usepackage{pdfpages}
\usepackage{siunitx}
\input{preamble_packages.tex}

\input{preamble_symbols.tex}

\input{shortcuts.tex}

\hypersetup{colorlinks, hypertexnames=false, pageanchor=true,
            linkcolor=blue, citecolor={green!50!black}, urlcolor=cyan,
            pdftitle={Factorization-free Orthogonal Projection onto the Positive Semidefinite Cone with Composite Polynomial Filtering}
            }
\makeatletter
\AddToHook{cmd/appendix/before}{\def\cref@section@alias{appendix}}
\makeatother
\mathtoolsset{centercolon}
\crefname{assumption}{Assumption}{Assumptions}

\theoremstyle{thmstyleone}%
\newtheorem{theorem}{Theorem}%

\theoremstyle{thmstyletwo}%
\newtheorem{remark}{Remark}%

\theoremstyle{thmstylethree}%
\newtheorem{definition}{Definition}%

\title{Factorization-free Orthogonal Projection onto the Positive Semidefinite Cone with Composite Polynomial Filtering}

\author{%
Shucheng Kang\thanks{School of Engineering and Applied Sciences, Harvard University. \texttt{\{skang1, hyhan, hankyang\}@g.harvard.edu}}%
\and Haoyu Han\footnotemark[1]%
\and Antoine Groudiev\footnotemark[1]\hspace{1.5mm}\thanks{Computer Science Department, Ecole Normale Supérieure, PSL University. \texttt{antoine.groudiev@ens.psl.eu}}%
\and Heng Yang\footnotemark[1]\hspace{1.5mm}\thanks{NVIDIA Corporation}%
}

\date{\today}

\begin{document}

\maketitle

\begin{abstract}
    We propose a factorization-free method for orthogonal projection onto the positive semidefinite (PSD) cone, leveraging composite polynomial filtering. Inspired by recent advances in homomorphic encryption, our approach approximates the PSD cone projection operator using a carefully optimized composite polynomial evaluated exclusively via matrix-matrix multiplications. This approach enables efficient GPU implementations with low-precision arithmetic, significantly outperforming the classical PSD cone projection using state-of-the-art GPU-based eigenvalue decomposition solvers. Specifically, our method achieves a consistent relative error of $10^{-3}$ in half-precision arithmetic with only 22 matrix-matrix multiplications, providing roughly a $10\times$ speed-up over NVIDIA's \cusolver{} routines on various large-scale matrices. In single-precision arithmetic with emulation on B200 GPUs, our approach maintains competitive accuracy while achieving up to a $2\times$ speed-up. Consequently, for a $10,000 \times 10,000$ dense symmetric matrix, our method requires approximately $55$ ms in half-precision and $400$ ms in single-precision arithmetic on B200 GPUs. Integration into a first-order semidefinite programming solver confirms that our low-precision projections reliably yield solutions of moderate accuracy. 
    
\end{abstract}

\input{sections/intro-new.tex}

\input{sections/related-works.tex}

\input{sections/method_new.tex}


\input{sections/evaluation.tex}
\input{sections/conclusion.tex}

\clearpage

\appendix

\input{sections/appendix.tex}

\bibliographystyle{plain}
\bibliography{references/refs.bib,references/myRefs.bib}
\end{document}

%% file: preamble_packages.tex
\usepackage{comment}
\usepackage{siunitx}
\usepackage{relsize}
\usepackage{ifthen}
\usepackage[colorinlistoftodos]{todonotes}

\usepackage[caption=false]{subfig}

\usepackage[vlined,ruled,linesnumbered]{algorithm2e}
\usepackage{graphics} 
\usepackage{rotating}
\usepackage{color}
\usepackage{enumerate}
\usepackage[T1]{fontenc}
\usepackage{psfrag}
\usepackage{epsfig} 
\usepackage{booktabs}
\usepackage{graphicx,url}
\usepackage{multirow}
\usepackage{array}
\usepackage{latexsym}
\usepackage{amsfonts}
\usepackage{amsmath}
\usepackage{amssymb}
\usepackage{mathtools}
\usepackage{xstring}
\usepackage{multirow}
\usepackage{xcolor}
\usepackage{prettyref}
\usepackage{flexisym}
\usepackage{bigdelim}
\usepackage{breqn} 
\usepackage{listings}

\usepackage{enumitem}
\usepackage{xspace}
\usepackage{bm}
\graphicspath{{./figures/}}
\usepackage{tikz}
\usetikzlibrary{matrix,calc}

\usepackage{mdwlist}

\makecompactlist{itemize}{stditemize}


%% file: preamble_symbols.tex
\newcommand{\cf}{\emph{cf.}\xspace}

\newcommand{\bdmath}{\begin{dmath}}
\newcommand{\edmath}{\end{dmath}}
\newcommand{\beq}{\begin{equation}}
\newcommand{\eeq}{\end{equation}}
\newcommand{\bdm}{\begin{displaymath}}
\newcommand{\edm}{\end{displaymath}}
\newcommand{\bea}{\begin{eqnarray}}
\newcommand{\eea}{\end{eqnarray}}
\newcommand{\beal}{\beq \begin{array}{ll}}
\newcommand{\eeal}{\end{array} \eeq}
\newcommand{\beas}{\begin{eqnarray*}}
\newcommand{\eeas}{\end{eqnarray*}}
\newcommand{\ba}{\begin{array}}
\newcommand{\ea}{\end{array}}
\newcommand{\bit}{\begin{itemize}}
\newcommand{\eit}{\end{itemize}}
\newcommand{\ben}{\begin{enumerate}}
\newcommand{\een}{\end{enumerate}}

\newcommand{\calA}{{\cal A}}

\newcommand{\calO}{{\cal O}}

\newcommand{\hide}[1]{}

\newcommand{\hiddenText}{{\color{gray} hidden text.}}
\newcommand{\hideWithText}[1]{\hiddenText}

\newcommand{\subject}{\text{ subject to }}

\DeclareMathOperator*{\argmin}{arg\,min}

\newcommand{\norm}[1]{\left\| #1 \right\|}

\newcommand{\tran}{^{\mathsf{T}}}

\newcommand{\Real}[1]{ { {\mathbb R}^{#1} } }

\newcommand{\blue}[1]{{\color{blue}#1}}

\newcommand{\linkToPdf}[1]{\href{#1}{\blue{(pdf)}}}
\newcommand{\linkToPpt}[1]{\href{#1}{\blue{(ppt)}}}
\newcommand{\linkToCode}[1]{\href{#1}{\blue{(code)}}}
\newcommand{\linkToWeb}[1]{\href{#1}{\blue{(web)}}}
\newcommand{\linkToVideo}[1]{\href{#1}{\blue{(video)}}}
\newcommand{\linkToMedia}[1]{\href{#1}{\blue{(media)}}}
\newcommand{\award}[1]{\xspace} 



%% file: shortcuts.tex
\newcommand{\Sn}{\mathbb{S}^n}

\renewcommand{\norm}[1]{\left\lVert #1 \right\rVert}
\newcommand{\inprod}[2]{\left\langle #1, #2 \right\rangle}

\newcommand{\sym}[1]{\mathbb{S}^{#1}}

\newcommand{\bmat}{\left[ \begin{array}}
\newcommand{\emat}{\end{array}\right]}
\newcommand{\psd}[1]{\sym{#1}_{+}}

\newcommand{\abs}[1]{\left|#1\right|}

\newcommand{\PiSnp}[1]{\Pi_{\psd{n}} (#1)}
\newcommand{\sign}{f_\mathsf{sign}}
\newcommand{\normF}[1]{\norm{#1}_\mathsf{F}}
\newcommand{\normtwo}[1]{\norm{#1}_2}
\newcommand{\frelu}{f_\mathsf{ReLU}}
\DeclareDocumentCommand{\diag}{o}{
  \IfNoValueTF{#1}{\mathrm{diag}}{\mathrm{diag} \left( #1 \right)}
}
\newcommand{\cusolver}{\texttt{cuSOLVER}}
\newcommand{\odd}{\mathsf{odd}}
\newcommand{\float}{\mathsf{float}}

\DeclareDocumentCommand{\fstar}{o}{
  \IfNoValueTF{#1}{f^\star}{f^\star_{#1}}
}
\DeclareDocumentCommand{\ftstar}{o}{
  \IfNoValueTF{#1}{\tilde{f}^\star}{\tilde{f}^\star_{#1}}
}
\newcommand{\Remez}{\mathrm{Remez}}

\newcommand{\doubleprec}{\texttt{FP64}}
\newcommand{\singleprec}{\texttt{FP32}}
\newcommand{\halfprec}{\texttt{FP16}}

\newcommand{\Asdp}{\calA}
\newcommand{\AsdpT}{\calA^*}

\newcommand{\Symp}[1]{\mathbb{S}_{+}^{#1}}

\newcommand{\Xk}{X^{(k)}}

\newcommand{\Sk}{S^{(k)}}

\newcommand{\Xkpo}{X^{(k+1)}}
\newcommand{\ykpo}{y^{(k+1)}}
\newcommand{\Skpo}{S^{(k+1)}}

\newcommand{\CSdouble}{\cusolver{} \doubleprec}
\newcommand{\CSsingle}{\cusolver{} \singleprec}
\newcommand{\NSsingle}{\texttt{Newton-Schulz} \singleprec}
\newcommand{\NShalf}{\texttt{Newton-Schulz} \halfprec}
\newcommand{\PEsingle}{\texttt{Polar} \texttt{Express} \singleprec}
\newcommand{\PEhalf}{\texttt{Polar} \texttt{Express} \halfprec}
\newcommand{\CPsingle}{\texttt{Composite} \singleprec}
\newcommand{\CPhalf}{\texttt{Composite} \halfprec}

\newcommand{\rebuttal}[1]{#1}

%% file: sections/intro-new.tex

\section{Introduction}
\label{sec:intro}

Consider the classical problem of orthogonally projecting a symmetric matrix \(X\) onto the positive semidefinite (PSD) cone:
\begin{equation}
    \label{eq:intro:psdcone-proj}
    \PiSnp{X} \;:=\; \argmin_{Y \in \psd{n}} \,\frac{1}{2}\,\normF{Y - X}^2,
\end{equation}
where \(\Sn\) denotes the space of \(n \times n\) real symmetric matrices, \(\psd{n} := \{A \in \Sn \,\mid\, A \succeq 0\}\) is the PSD cone, and \(\normF{\cdot}\) is the Frobenius norm.  
Computing this projection---yielding the nearest PSD approximation of \(X\)---is a fundamental primitive in numerous fields, including optimization \cite{wen10mpc-admmsdp,kang25arxiv-admm}, signal processing \cite{yang23eurasip-structured-covariance-estimation}, finance \cite{higham02ima-nearest-correlation-finance}, machine learning \cite{weinberger09jmlr-distance-metric-learning}, and quantum mechanics \cite{drusvyatskiy15qip-projection-quantum-channel-construction,barberarodriguez25prr-projective-methods-quantum-process}. In particular, it is invoked repeatedly to enforce feasibility in first-order algorithms for large-scale semidefinite programming (SDP) \cite{zheng17ifac-cdcs-sdpsolver,odonoghue23-scs-sdpsolver,kang24wafr-strom,yang23mp-stride,yang15mpc-sdpnalplus-sdpsolver,garstka21jota-cosmo}.

\paragraph{Classical factorization-based solution.}
Suppose the spectral decomposition of $X$ is given by
\begin{equation}\label{eq:X-spectral-decomposition}
    X = Q \Lambda Q\tran,
\end{equation}
where $\Lambda = \diag(\lambda_{1}, \dots, \lambda_{n})$ is a diagonal matrix containing the $n$ real eigenvalues of $X$, and $Q \in \Real{n \times n}$ is an orthonormal matrix whose columns are the corresponding eigenvectors. Higham \cite{higham88laa-computing-psdcone-projection} gave a closed-form solution to \eqref{eq:intro:psdcone-proj} as 
\begin{align}
    \label{eq:intro:closed-form}
    \PiSnp{X} = Q \diag[\max\{\lambda_1,0\}, \dots, \max\{\lambda_n,0\}] Q\tran.
\end{align}
Essentially, projecting \(X\) onto the PSD cone amounts to zeroing out its negative eigenvalues and corresponding eigenvectors. This insight naturally leads to using the eigenvalue decomposition (EVD) as the standard method for PSD cone projection. Due to its robustness and numerical reliability, full EVD has become the default routine for PSD cone projection and widely used in first-order SDP optimizers \cite{wen10mpc-admmsdp,weinberger09jmlr-distance-metric-learning,odonoghue23-scs-sdpsolver,yang23eurasip-structured-covariance-estimation,kang24wafr-strom}.

However, recent advances in numerical optimization---particularly first-order methods for solving large-scale SDPs derived from convex relaxations~\cite{yang23mp-stride,yang15mpc-sdpnalplus-sdpsolver,kang24wafr-strom,kang25rss-spot,magron23book-sparse,klep24mp-state,wang24prx-certifying,han25rss-xm}---require repeated projections onto the PSD cone for dense matrices with dimensions in the \emph{tens of thousands} (i.e., $n \geq 10,000$ in \eqref{eq:intro:psdcone-proj}). As a result, classical projection via full EVD has become suboptimal for two reasons. 
First, the tridiagonalization phase of QR-factorization-based EVD is memory-bound \cite{haidar12-magma,ding18mp-spectral-operator},
limiting the ability to leverage high-throughput modern hardware such as the graphics processing unit (GPU) for acceleration. Second, high-precision projections (e.g., 64-bit machine precision) can be unnecessary in first-order methods, where standard termination criteria only require the Karush-Kuhn-Tucker (KKT) residual to fall below \(10^{-2} \sim 10^{-4}\)~\cite{odonoghue23-scs-sdpsolver,zheng17ifac-cdcs-sdpsolver,kang24wafr-strom}. Yet implementing EVD in lower precision (e.g., half precision on GPU Tensor Cores) remains difficult due to the strict orthogonality requirements intrinsic to factorization-based algorithms~\cite{yang21sisc-rounding-error-householder}. Therefore, the central motivation of this paper is:
\vspace{2mm}
\begin{quotation}
    \textit{Can we design an efficient, factorization-free algorithm for PSD cone projection that is compatible with low-precision arithmetic? Moreover, how much speedup can such an approach achieve compared to state-of-the-art GPU-based eigenvalue-decomposition solvers?}
\end{quotation}
\vspace{2mm}

\paragraph{Factorization-free solution via polynomial filtering.}
A key step toward a factorization-free approach for PSD cone projection is recognizing that the closed-form solution~\eqref{eq:intro:closed-form} can be expressed as a \emph{spectral operator} applied to \(X\). Specifically, a spectral operator \(F(X): \Sn \to \Sn\) is said to be generated by a scalar function \(f(x): \mathbb{R} \to \mathbb{R}\) if $F(X)$ is equivalent to applying $f(x)$ to the eigenvalues of $X$~\cite{ding18mp-spectral-operator}, namely,
\begin{equation}
    \label{eq:spectral-operator}
    F(X) = Q \diag(f(\lambda_1), \dots, f(\lambda_n)) Q\tran,
\end{equation}
where \(X\) admits the spectral decomposition defined in~\eqref{eq:X-spectral-decomposition}. Under this framework, the PSD projection \(\PiSnp{X}\) corresponds to the spectral operator generated by the Rectified Linear Unit (ReLU): \(\frelu(x) = \max\{x, 0\}\).

\rebuttal{
    Building on this observation, one can approximate $\frelu(x)$ by a polynomial $p(x)$. Throughout the paper, unless stated otherwise, scalar approximation errors are measured in the uniform norm on the scaled interval $[-1,1]$. This is the relevant interval because, before applying the polynomial filter to a matrix, we rescale the matrix by an upper bound on its spectral norm so that all eigenvalues of the scaled matrix lie in $[-1,1]$. 
}
The advantage of a polynomial approximation is that it avoids explicit eigenvalue decomposition:
\begin{equation}
    \label{eq:polynomial-filter-overview}
    \PiSnp{X} = Q \diag[\frelu(\lambda_1), \dots, \frelu(\lambda_n)] Q\tran
    \approx Q \diag[p(\lambda_1), \dots, p(\lambda_n)] Q\tran = p(X),
\end{equation}
where the final equality follows from the fact that for any integer \(d \in \mathbb{N}\),
\[
\resizebox{\textwidth}{!}{$
X^d = \underbrace{(Q \diag[\lambda_1, \dots, \lambda_n] Q\tran)(Q \diag[\lambda_1, \dots, \lambda_n] Q\tran)\cdots(Q \diag[\lambda_1, \dots, \lambda_n] Q\tran)}_{d \text{ times}}  = Q \diag[\lambda_1^d, \dots, \lambda_n^d] Q\tran.
$}
\]
Thus, a good polynomial approximation \(p(x) \approx \frelu(x)\) enables PSD projection purely through general matrix-matrix multiplications (GEMM), making the method highly amenable to GPU acceleration \cite{nvidiaA100,nvidia24blog-cublas}. This technique is commonly referred to as \emph{polynomial filtering}~\cite{francisco17laa-fixedpoint-approximate-psdcone}.
Under this paradigm, the central challenge becomes:

\vspace{2mm}
\begin{quotation}
    \textit{Given a fixed budget of GEMMs, how to construct an optimal polynomial \(p(x)\) that best approximates \(\frelu(x)\), and how to efficiently implement the resulting polynomial filter on GPUs?}
\end{quotation}
\vspace{2mm}

\rebuttal{
    If $p(x)$ is allowed to range over the full degree-$d$ polynomial space, then the scalar minimax approximation problem for $\frelu$ on $[-1,1]$ can be solved offline by classical Chebyshev/Remez exchange methods \cite[Theorem 10.1]{trefethen19book-approximation}. However, this scalar preprocessing efficiency should not be confused with fast matrix evaluation. Because $\frelu(x)=(x+\abs{x})/2$ has a kink at the origin, the best degree-$d$ uniform polynomial approximation error on $[-1,1]$ is only of order $1/d$. Therefore, high accuracy requires a high degree $d$, and evaluating an unstructured degree-$d$ polynomial on a dense matrix generally requires $\calO(d)$ GEMMs. To compete with EVD in runtime, we instead seek a structured high-degree polynomial that achieves a better trade-off between approximation quality and matrix-multiplication cost.
}

\paragraph{Contribution.}
Inspired by recent work in homomorphic encryption on approximating the \emph{sign function}~\cite{lee22tdsc-minimax-approximation-sign}, we propose a \emph{composite polynomial filter} for approximating $\frelu(x)$. The idea is to express the approximation as a composition of $T$ low-degree polynomials:
\begin{equation}
    \label{eq:composite-polynomial-filter}
    p(x) = f_T \circ f_{T-1} \circ \dots \circ f_1(x),
\end{equation}
where each $f_i$, for $i \in [T]$, has degree at most $d_i$. \rebuttal{The composite form is motivated by matrix-evaluation cost, not by a better approximation rate in the classical degree-counting sense. Although $p$ can have total degree as large as $\prod_{i=1}^T d_i$, we evaluate $p(X)$ by applying the low-degree filters $f_1,\ldots,f_T$ sequentially, without expanding the composition. This requires roughly $\sum_{i=1}^T d_i$ GEMMs rather than $\prod_{i=1}^T d_i$ GEMMs. Thus, for dense matrix inputs, the relevant budget is the number of GEMMs, not only the formal scalar degree.} This makes it possible to realize high-degree polynomial filters at low matrix-multiplication cost, which is well suited for GPU implementation.

We contribute (\emph{i}) an algorithmic framework for optimizing composite polynomial coefficients to minimize the worst-case approximation error of $\frelu(x)$, and (\emph{ii}) a high-performance GPU implementation of the resulting filter that outperforms classical factorization-based PSD projection. Specifically:

\begin{enumerate}
    \item \textbf{Optimization of polynomial coefficients.}  
    To minimize the worst-case approximation error under a fixed budget of GEMMs, we introduce a two-stage framework for designing the composite polynomial filter:
    \begin{enumerate}
        \item[(i)] \emph{Minimax initialization.}  
        We construct a sequence of composite polynomials that approximate the matrix sign function---a surrogate for the ReLU function. The coefficients of each low-degree component are computed via the Remez algorithm~\cite{remez62ssr-remez}, following the approach in~\cite{lee22tdsc-minimax-approximation-sign}. The obtained composite polynomials are proven to be minimax optimal for the sign function under mild conditions.
        
        \item[(ii)] \emph{Gradient-based refinement.}  
        Since the minimax approximation is tailored to the sign function and not directly to $\frelu(x)$, we further refine the coefficients using gradient descent. This step treats the polynomial coefficients as unconstrained, non-convex optimization variables, directly targeting approximation of \(\frelu(x)\). The refinement further reduces error beyond the initial minimax stage.
    \end{enumerate}
    As a result of the two-stage optimization, our method achieves a lower approximation error with the same number of GEMMs compared to prior polynomial filtering approaches for PSD projection~\cite{francisco17laa-fixedpoint-approximate-psdcone}.

    \item \textbf{GPU implementation and performance evaluation.}  
    We implemented the composite polynomial filter in native C++ and CUDA, leveraging Tensor Cores and \texttt{BF16x9} emulation techniques~\cite{cublas129} to achieve high computational efficiency. We benchmarked our method on both NVIDIA B200 and NVIDIA H100 GPUs, comparing its performance against the commercial \cusolver{} library~\cite{cusolver} and previous polynomial filtering approaches~\cite{amsel25arxiv-polar-express,francisco17laa-fixedpoint-approximate-psdcone}. 
    In half-precision arithmetic, our filter consistently attains a relative error of approximately $10^{-3}$ using only $22$ GEMM operations across diverse large-scale matrix datasets. This results in a stable speed-up of approximately $10\times$ compared to \cusolver's eigenvalue decomposition (EVD). For single-precision arithmetic, with emulation enabled on the B200 GPU, our polynomial filter achieves up to $2\times$ speed-up over \cusolver's single-precision EVD implementation while utilizing $31$ GEMM launches, and it maintains comparable accuracy.
    Furthermore, when integrated into a first-order SDP solver, our lower-accuracy projection method remains effective, allowing the solver to reliably converge to solutions of moderate accuracy.

\end{enumerate}

Our PSD cone projection package based on composite polynomial filtering is open-sourced at

\vspace{2mm}
\begin{center}
    \url{https://github.com/ComputationalRobotics/psd_projection}.
\end{center}
\vspace{2mm}

\paragraph{Outline.} Related work is reviewed in Section~\ref{sec:related-work}. Section~\ref{sec:comp} introduces our composite polynomial filter design framework. Implementation details and experimental results are presented in Section~\ref{sec:evaluation}, followed by conclusions in Section~\ref{sec:conclusion}.

%% file: sections/related-works.tex

\section{Related Work}
\label{sec:related-work}

\paragraph{PSD cone projection via eigenvalue decomposition.} Eigenvalue decomposition is often employed to obtain $Q$ and $\{ \lambda_i \}_{i=1}^n$ in~\eqref{eq:intro:closed-form}. In general, the full decomposition of an $n \times n$ dense symmetric matrix $X$ has computational complexity $\mathcal{O}(n^3)$ and memory complexity $\mathcal{O}(n^2)$ and is widely implemented across various computing architectures~\cite{anderson99siam-lapack,cusolver,tomov10pc-magma}. Almost all existing first-order SDP solvers rely on full eigenvalue decomposition for PSD cone projection~\cite{zheng17ifac-cdcs-sdpsolver,odonoghue23-scs-sdpsolver,kang24wafr-strom,yang15mpc-sdpnalplus-sdpsolver,garstka21jota-cosmo}. When either the positive or negative eigenvalue component is low-rank, PSD cone projection can be accelerated using low-rank eigenvalue decomposition algorithms~\cite{rontsis22jota-approximate-admm}. Sketching- and randomization-based methods for approximate large-scale PSD cone projection also exist~\cite{jones24arxiv-random-psdcone}, but their convergence guarantees hold only in a probabilistic sense. All these methods explicitly rely on knowledge of (at least some) eigenvalue-eigenvector pairs.

\rebuttal{
    \paragraph{Projection methods for conic optimization.}
    Projection operators are central in conic feasibility and conic optimization algorithms, such as projection-based approaches for semidefinite least-squares problems~\cite{malick04simaa-dual-sdls}, conic feasibility and polynomial sum-of-squares decompositions~\cite{henrion11oms-conic-feasibility-sos}, and regularization methods for semidefinite programming~\cite{malick09siopt-regularization-sdp}; see also~\cite{henrion12handbook-projection-conic}. Our work is complementary: we focus on accelerating the elementary projection onto $\mathbb{S}^n_+$ itself, replacing spectral decomposition with a factorization-free composite polynomial filter evaluated by dense matrix-matrix multiplications.
}
    
\paragraph{Matrix sign function approximation via polynomial/rational fixed point iteration.} Let $X$ admit the spectral decomposition in \eqref{eq:X-spectral-decomposition}, the matrix sign function is defined as \cite{denman76amc-matrix-sign}
\begin{equation}\label{eq:matrix-sign}
    \sign(X):=Q \diag[\sign(\lambda_1),\dots,\sign(\lambda_n)] Q\tran,
\end{equation}
where $\sign(\cdot)$ is defined in \eqref{eq:comp:sign}. A substantial body of literature has developed polynomial- and rational-based fixed-point schemes for computing the matrix sign function, including the Newton iteration~\cite{higham08book-functions-matrices,higham86jssc-polar-decomposition}, the inverse Newton iteration~\cite{nakatsukasa10simaa-halley-polar}, and several Newton--Schulz variants~\cite{chen14anl-stable-newton-schulz,bjorc71sina-iterative-estimate-orthogonal}. These methods involve only matrix-matrix multiplications or matrix inversions, making them more suitable for large-scale parallel architectures where the computational and communication overhead of full eigenvalue decomposition quickly becomes prohibitive. In the context of PSD cone projection, \cite{francisco17laa-fixedpoint-approximate-psdcone} proposed rewriting~\eqref{eq:intro:closed-form}---the spectral operator generated by $\frelu(x)$---as $0.5 X (I + \sign(X))$ and employed a variant of the Newton-Schulz iteration to approximate the matrix sign function, primarily targeting PSD projection of sparse, banded matrices.
However, while fixed-point iterations can achieve local quadratic convergence given knowledge of the smallest absolute eigenvalue~\cite{francisco17laa-fixedpoint-approximate-psdcone}, in practice, achieving even moderate accuracy typically requires many iterations~\cite{chen14anl-stable-newton-schulz}, thus diminishing their speed advantage relative to full eigenvalue decomposition. In optimization of neural networks, \cite{amsel25arxiv-polar-express} recently employed a sequence of degree-5 minimax composite polynomials to approximate the polar factor (a generalization of matrix sign function) at low accuracy. Nevertheless, the GPU-level competitiveness of such composite polynomial filtering, relative to state-of-the-art factorization-based approaches, remains unclear.

\paragraph{Background for GPU computation.} GPU math libraries organize kernels using the BLAS hierarchy:
Level-2 routines (matrix-vector) involve $\mathcal{O}(n^2)$ floating-point operations and $\mathcal{O}(n^2)$ data movement, yielding low arithmetic intensity; Level-3 routines (matrix-matrix) perform $\mathcal{O}(n^3)$ work for the same $\mathcal{O}(n^2)$ memory traffic, achieving far higher intensity~\cite{cublas129}. Under the Roofline model\cite{williams09acm-roofline}, this means Level-2 kernels are typically memory-bound while Level-3 kernels can become compute-bound once their arithmetic intensity exceeds the machine's balance point. Empirical studies confirm that large-scale GEMM kernels on modern GPUs routinely sustain over 90\% of peak compute throughput, making them archetypal compute-bound workloads \cite{volkov08acm-benchmark-gpu-dla}. In contrast, dense eigenvalue decomposition (EVD) begins with a reduction to (banded) tridiagonal form dominated by BLAS-2 updates; because these updates are memory-bound, the end-to-end EVD pipeline is usually limited by memory bandwidth even though it contains some compute-heavy BLAS-3 sweeps. Furthermore, modern GPUs such as NVIDIA's Ampere A100 incorporate Tensor Cores---specialized mixed-precision matrix-multiply-accumulate (MMA) units that push peak throughput beyond 300 TFLOP/s in half precision and deliver orders-of-magnitude speedups for dense linear-algebra kernels compared to general CUDA cores~\cite{nvidiaA100}. This further amplifies BLAS-3 operator's advantage over BLAS-2 operator.

%% file: sections/method_new.tex

\section{Composite Polynomial Filter Design}
\label{sec:comp}

In this section we present a two-stage framework that efficiently approximates $\frelu(x)$ with a sequence of low-degree polynomials combined through addition, multiplication, and composition. 

\paragraph{Minimax formulation for $\frelu(x)$} Recall from Section \ref{sec:intro} that the goal is to find a composite polynomial filter, in the form of \eqref{eq:composite-polynomial-filter}, that best approximates $\frelu(x)$. To do so, we adopt the classical minimax criterion from approximation theory~\cite[Chapter 10]{trefethen19book-approximation}:
\begin{eqnarray}
    \label{eq:comp:minimax-relu}
    \inf_{f_1,\dots,f_T}\;
        \max_{x \in [-1,1]}
        & \abs{f_{T} \circ f_{T-1} \circ \cdots \circ f_{1}(x) - \frelu(x)} \\[4pt]
    \subject & f_t \in \mathbb{R}_{d_t}[x], \; t = 1,\dots,T. \nonumber
\end{eqnarray}
Here the total number of polynomials $T$ and their degrees $\{d_t\}_{t=1}^T$ are fixed in advance as the computational budget. 
The set $\mathbb{R}_{d_t}[x]$ includes real (univariate) polynomials of degree at most~$d_t$.  
Restricting $x$ to $[-1,1]$ does not result in loss of generality, as we can efficiently rescale the matrix~$X$ so that its eigenvalues lie in this interval. Since the minimax optimality under composition may not be attainable~\cite{girosi1990bc-newworks-best-approx}, we use ``$\inf$'' in~\eqref{eq:comp:minimax-relu}. Let $D:= \prod_{t=1}^T d_t$ be the degree of the composite polynomial filter.

\paragraph{Challenges} Directly solving~\eqref{eq:comp:minimax-relu} is challenging, even under the assumption that the minimax optimum of the composite structure is attainable. Classical approximation tools such as Chebyshev criteria and the Remez algorithm~\cite{remez62ssr-remez} are not applicable, as the composite polynomial $f_T \circ \cdots \circ f_1(x)$ does not span the full polynomial space $\mathbb{R}_D[x]$. While global optimization techniques---such as the two-level semidefinite relaxation hierarchy for robust polynomial optimization~\cite{lasserre2011jgo-minmax-robost-pop}---can, in principle, solve~\eqref{eq:comp:minimax-relu}, they face severe scalability limitations due to the high polynomial degrees involved.
\rebuttal{In contrast, rational approximation is attractive from a purely scalar approximation viewpoint. Rational approximants to functions with a nonsmooth point, such as $\abs{x}$ and hence $\frelu(x)=(x+\abs{x})/2$, can achieve root-exponential convergence in the degree~\cite{newman64um-rational-approximation-abs}. However, evaluating a rational filter on a matrix requires operations such as $p(X)q(X)^{-1}$, or shifted linear solves after partial-fraction expansion. These operations introduce inversions, factorizations, or iterative solves, and therefore fall outside our target model of a factorization-free, GEMM-only implementation that exploits low-precision Tensor Cores.}

\paragraph{Two-stage design strategy} To address these challenges, we adopt a two-stage design strategy to approximately solve~\eqref{eq:comp:minimax-relu}. In Stage I (Section~\ref{sec:sec:stage-1}), we construct a minimax-optimal composite polynomial approximation for the surrogate function $\sign(x)$. In Stage II (Section~\ref{sec:sec:stage-2}), we locally refine the polynomials obtained in Stage I using the true minimax objective for $\frelu(x)$ in~\eqref{eq:comp:minimax-relu}.

\subsection{Stage I: Minimax Composite Polynomials for \texorpdfstring{$\sign(x)$}{sign(x)}}
\label{sec:sec:stage-1}

\paragraph{Minimax formulation for $\sign(x)$} 
\rebuttal{For the direct composite $\frelu$ problem~\eqref{eq:comp:minimax-relu}, we are not aware of an efficient greedy algorithm that computes minimax composite coefficients with optimality guarantees. On the other hand, approximating $\sign(x)$ with polynomial compositions admits a more tractable sequential construction~\cite{chen14anl-stable-newton-schulz,lee22tdsc-minimax-approximation-sign,amsel25arxiv-polar-express,francisco17laa-fixedpoint-approximate-psdcone}. The sign surrogate has two simplifying features. First, $\sign(x)$ is odd, so it is natural to restrict the component filters to odd polynomials. Second, away from the discontinuity at the origin, approximating $\sign(x)$ reduces on the positive interval $[\epsilon,1]$ to approximating the constant function $1$.}
Consider
\begin{eqnarray}
    \label{eq:comp:minimax-sign}
    \left\{ \fstar[t] \right\}_{t=1}^T = 
    \argmin_{f_1,\dots,f_T}\;
        \max_{x \in [-1, -\epsilon] \cup [\epsilon, 1]}
        & \abs{f_{T} \circ f_{T-1} \circ \cdots \circ f_{1}(x) - \sign(x)} \\[4pt]
    \subject & f_t \in \mathbb{R}^{\odd}_{d_t}[x], \; t = 1,\dots,T, \nonumber
\end{eqnarray}
where $\epsilon \in (0, 1)$ is a tunable hyper-parameter to address the discontinuity issue when approximating $\sign(x)$, $\mathbb{R}^{\odd}_{d_t}[x]$ represents the subspace of $\mathbb{R}_{d_t}[x]$ containing all odd polynomials, and 
\begin{align}
    \label{eq:comp:sign}
    \sign(x) := \begin{cases}
        1,   & x > 0, \\[2pt]
        0.5, & x = 0, \\[2pt]
       -1,   & x < 0,
    \end{cases}
    \quad \text{with} \quad \frelu(x) = \frac{1}{2} x (1 + \sign(x)).
\end{align}
We only consider odd polynomials in~\eqref{eq:comp:minimax-sign} since $\sign(x)$ is an odd function. 
\cite[Definition 7]{lee22tdsc-minimax-approximation-sign} proposes an efficient greedy algorithm to solve~\eqref{eq:comp:minimax-sign}, presented in Algorithm~\ref{alg:comp:seq-remez}. \rebuttal{The key simplification is that the sign formulation decomposes the coupled composite design into $T$ one-dimensional Remez subproblems, one for each component polynomial, each solved over an odd-polynomial Haar space.}
In Algorithm~\ref{alg:comp:seq-remez}, \eqref{eq:comp:remez-1} holds since $\sign$ and $f$ are both odd.
$\Remez ( g, \ [a, b], \ \Phi  )$ in~\eqref{eq:comp:remez-2} is used to determine the unique minimax approximate polynomial for a continuous function $g$ over interval $[a, b]$ with basis $\Phi$. Over the iterations, $a_t$ remains positive. Thus, $\mathbb{R}^{\odd}_{d_t}[x]$ satisfies the Haar condition on $[a_t, b_t]$ and the Remez algorithm converges~\cite[Lemma A.1]{amsel25arxiv-polar-express}.  We defer the detailed discussion for the Remez algorithm to Appendix~\ref{app:remez}. The next Theorem states the correctness of Algorithm~\ref{alg:comp:seq-remez}. 




\begin{algorithm}[htbp]
    \caption{A sequential Remez algorithm to solve~\eqref{eq:comp:minimax-sign}}
    \label{alg:comp:seq-remez}
\begin{algorithmic}[1]      
    \REQUIRE Composite step $T$, degrees $\{ d_t \}_{t=1}^T$, hyperparameter $\epsilon$.
    \ENSURE  Minimax composite polynomials $\{ \fstar[t] \}_{t=1}^T$ for~\eqref{eq:comp:minimax-sign}.
    \STATE Initialization: $a_t \leftarrow \epsilon, \ b_t \leftarrow 1$.

    \FOR{$t = 1$ \textnormal{to} $T$}
        \STATE (1) Choose a basis $\Phi_t$ of $\mathbb{R}^{\odd}_{d_t}[x]$. $\fstar[t]$ is the minimax approximate polynomial of degree at most $d_t$ for $\sign$ over $[-b_t, -a_t] \cup [a_t, b_t]$:
        \STATE
        \begin{align}
            \fstar[t] := & \argmin_{f \in \mathbb{R}^{\odd}_{d_t}[x]} \max_{x \in [-b_t, -a_t] \cup [a_t, b_t]} \abs{f(x) - \sign(x)} \nonumber \\
            = & \argmin_{f \in \mathbb{R}^{\odd}_{d_t}[x]} \max_{x \in [a_t, b_t]} \abs{f(x) - \sign(x)} \label{eq:comp:remez-1} \\
            \leftarrow & \Remez \left( \sign, \ [a_t, b_t], \ \Phi_t \right). \label{eq:comp:remez-2}
        \end{align}

        \STATE (2) Set the $(t+1)$'th interval $[a_{t+1}, b_{t+1}]$ as the image of $[a_{t}, b_{t}]$ under $\fstar[t]$:
        \STATE
        \begin{align}
            \label{eq:comp:update-interval}
            [a_{t+1}, b_{t+1}] \leftarrow \fstar[t] \left( [a_{t}, b_{t}] \right).
        \end{align}
    \ENDFOR
\end{algorithmic}
\end{algorithm}

\begin{theorem}[Minimax Optimality]
    \label{thm:comp:minimax}
    Given the composite step $T$, degrees $\{ d_t \}_{t=1}^T$, and hyperparameter $\epsilon$, $\{ \fstar[t] \}_{t=1}^T$ generated by Algorithm~\ref{alg:comp:seq-remez} is minimax optimal for~\eqref{eq:comp:minimax-sign}. 
\end{theorem}
Theorem~\ref{thm:comp:minimax} can be directly deduced from~\cite[Theorem 2]{lee22tdsc-minimax-approximation-sign}.~\cite[Theorem 4.1]{amsel25arxiv-polar-express} also provides an independent proof. We provide the proof here for completeness. 

\begin{proof}[Proof for Theorem~\ref{thm:comp:minimax}]
    Prove by contradiction. Denote $\{ \fstar[t] \}_{t=1}^T$ as the polynomial sequence generated by Algorithm~\ref{alg:comp:seq-remez}. Its approximation error is: 
    \[
        E^\star = \max_{x \in [\epsilon, 1]} \abs{\fstar[T] \circ \fstar[T-1] \circ \cdots \circ \fstar[1](x) - \sign(x)}.
    \] 
    Now suppose~\eqref{eq:comp:minimax-sign}'s optimal minimax value $E_{\min} < E^\star$. 
    Invoking~\cite[Theorem 2]{lee22tdsc-minimax-approximation-sign}, there exists a polynomial sequence $\{ w_t \}_{t=1}^T$ generated by Algorithm~\ref{alg:comp:seq-remez}, such that 
    \[
    E_{\min} \ge \max_{x \in [\epsilon, 1]} \abs{w_{T} \circ w_{T-1} \circ \cdots \circ w_1(x) - \sign(x)}.
    \]
    On the other hand, the uniqueness of each Remez algorithm step in~\eqref{eq:comp:remez-2} under Haar condition is guaranteed. Thus, $w_1 = \fstar[1], w_2 = \fstar[2], \dots, w_T = \fstar[T]$, which leads to the contradiction. 
\end{proof}



\subsection{Stage II: Gradient-based Refinement for \texorpdfstring{$\frelu(x)$}{f_ReLU(x)}}
\label{sec:sec:stage-2}

Given an approximation sequence $\{f_t^\star\}_{t=1}^T$ to $\sign(x)$, an approximation to $\frelu(x)$ follows immediately:
\begin{align}
    \label{eq:comp:fstar}
    \frelu(x) \;\approx\;
        \frac{1}{2}\,x\!\left(1 + \fstar[T] \circ \fstar[T-1] \circ \cdots \circ \fstar[1](x) \right)
        \;=:\;\fstar (x).
\end{align}
\rebuttal{This gives a structured initialization for the $\frelu$ problem~\eqref{eq:comp:minimax-relu}.}
However, even if $f_T^\star \circ f_{T-1}^\star \circ \cdots \circ f_1^\star(x)$ is minimax-optimal for~\eqref{eq:comp:minimax-sign}, there is no guarantee that the resulting $f^\star(x)$ is minimax-optimal for~\eqref{eq:comp:minimax-relu}. We therefore carry out an additional gradient-based local optimization of the coefficients, following the coefficient-refinement strategy of \cite{cesista2025muonoptcoeffs} and exploiting modern automatic-differentiation toolchains. Concretely, we minimize the loss function 
\begin{equation}\label{eq:loss-gradient}
    \ell(f_T,\dots,f_1):=\max_{x \in [-1, 1]} \abs{
        \frac{1}{2} x (1 + f_T \circ f_{T-1} \circ \cdots \circ f_1(x) ) - \frelu(x)},
\end{equation}
evaluating it on a dense set of sample points in the interval \([-1,1]\). The weights of \(\{f_t\}_{t=1}^T\) are initialized with the coefficients of \(\{\fstar[t]\}_{t=1}^T\) computed from Stage~I and are then progressively refined by gradient descent. Once this optimization converges, we denote the resulting composite polynomial by
\begin{align}
    \label{eq:comp:fstar-tilde}
    \ftstar(x) := \frac{1}{2}\,x\!\left(1 + \ftstar[T] \circ \ftstar[T-1] \circ \cdots \circ \ftstar[1](x) \right).
\end{align}
\rebuttal{We emphasize that Stage~II keeps the coefficients in composite form: it updates the component filters $f_1,\ldots,f_T$ without expanding their composition into a full degree-$D$ basis. This is different from directly applying Remez to $\frelu$ over the full degree-$D$ polynomial space, which would produce an unstructured polynomial whose matrix evaluation generally requires $\calO(D)$ GEMMs.}

\paragraph{Approximation error evaluation} Given an arbitrary composite function $f$, it is difficult to obtain the exact approximation error in the $\infty$-norm, $e(f) := \max_{x\in [-1, 1]} \abs{f(x) - \frelu(x)}$.
Since our implementation ultimately targets single- or even half-precision arithmetic, it is practically sufficient to evaluate the point-wise error at every single-precision representable value in $[-1,1]$.  
The complete set of such points, denoted $S_{\float}$, contains $2{,}130{,}706{,}433$ elements.  We therefore work with and evaluate the discretized error
\begin{align}
    \label{eq:comp:error-approx}
    e_{\float}(f) := \max_{x_i \in S_{\float}} \abs{f(x_i) - \frelu(x_i)} .
\end{align}
Although $e_{\float}(f)$ is merely an approximation of $e(f)$, our extensive benchmarking of PSD-cone projections (\cf \S~\ref{sec:sec:eva:psdcone}) clearly demonstrates the practical benefits of our refined coefficients.

\subsection{Explicit Formulations of Composite Polynomial Filters}

Integrating the techniques outlined above, we design two composite-polynomial variants, each tailored to single- and half-precision settings, respectively. Guided by the principles in~\cite{lee22tdsc-minimax-approximation-sign,amsel25arxiv-polar-express}, we set $d_t = 5$ for all $t$.

(1) \textit{Single precision setting.} We set $\epsilon = 10^{-3}$, $T = 10$. $\{ \fstar[t] \}_{t=1}^T$, $\{ \ftstar[t] \}_{t=1}^T$, and their corresponding $e_\float$ are listed in Table~\ref{tab:comp:single}. The final $\ftstar_{\mathsf{single}}$ involves $31$ GEMMs.

\input{tables/single_precision.tex} 

(2) \textit{Half precision setting.} We set $\epsilon = 10^{-3}$, $T = 7$. $\{ \fstar[t] \}_{t=1}^T$, $\{ \ftstar[t] \}_{t=1}^T$, and their corresponding $e_\float$ are listed in Table~\ref{tab:comp:half}. The final $\ftstar_{\mathsf{half}}$ involves $22$ GEMMs.

\input{tables/half_precision.tex}

\subsection{Practical Considerations for Finite-Precision Implementation}
\label{sec:comp:finite-prec}

When porting $\ftstar$ to GPUs that operate in single- or half-precision, two practical issues arise.

\paragraph{1. Digit loss during matrix rescaling}
Before applying $\ftstar$ to a matrix $X \in \Sn$, we rescale its spectrum to lie in $[-1,1]$, dividing $X$ by an upper bound on its spectral norm $\normtwo{X}$.  A common choice is the Frobenius norm $\normF{X}$~\cite{amsel25arxiv-polar-express} or the infinity norm $\norm{X}_\infty$~\cite{francisco17laa-fixedpoint-approximate-psdcone}.  However, since
\[
    \normtwo{X} \le \normF{X} \le \sqrt{n}\,\normtwo{X}, 
    \qquad
    \normtwo{X} \le \norm{X}_\infty \le \sqrt{n}\,\normtwo{X},
\]
for $n \ge 10{,}000$ these conservative bounds can cost the loss of $\log_{10}(\sqrt{n}) \ge 2$ decimal digits in the worst case.  
To tighten the bound, we adopt the following result, derived from the Residual Norm Theorem~\cite[Theorem 4.5.1]{parlett1987book-symmetric-eig}.

\begin{theorem}[Tighter upper bound for $\normtwo{A}$]
    \label{thm:comp:upper-bound}
    Let $A \in \Sn$ and let $0 \le \lambda_n \le \cdots \le \lambda_1$ be the eigenvalues of $A^2$.  
    Fix $\sigma \in \mathbb{R}$ and assume $\lambda_1$ is the eigenvalue of $A^2$ closest to $\sigma$ (i.e.\ $\lambda_1 = \argmin_{\lambda_i} \abs{\sigma - \lambda_i}$).  
    Then, for every unit vector $q \in \Real{n}$,
    \begin{align}
        \label{eq:comp:upper-bound}
        \normtwo{A} \le \sqrt{\sigma + \normtwo{A^2 q - \sigma q}} .
    \end{align}
\end{theorem}

\begin{proof}
    From~\cite[Theorem 4.5.1, Proof 2]{parlett1987book-symmetric-eig},
    \begin{align}
        \label{eq:comp:residual-norm-theorem}
        \abs{\lambda_1 - \sigma} \le \normtwo{A^2 q - \sigma q}.
    \end{align}
    Because $\normtwo{A} = \sqrt{\normtwo{A^2}} = \sqrt{\lambda_1}$, the claim follows.
\end{proof}

In practice, we choose $(\sigma,q)$ by taking the largest Ritz value produced by a 20-step Lanczos~\cite[Chapter 13]{parlett1987book-symmetric-eig} run on $A^2$ together with its Ritz vector. Although the theorem assumes $\sigma$ is the closest to $\lambda_{\max}(A^2)$, we have found this heuristic to yield a valid upper bound in all our extensive experiments. Since Lanczos iterations only involves matrix-vector multiplications, its running time is negligible compared to GEMMs.

\paragraph{2. Numerical instability}
As noted in~\cite[Section 4.4]{amsel25arxiv-polar-express}, minimax composite polynomials can be unstable in half-precision. Following their rescaling strategy, we multiply $X$ by $\frac{1}{1.01}$ at the end of every half-precision iteration. For single precision, we rescale by $\frac{1}{1.001}$ after each of the first eight iterations.

\paragraph{Final run-time algorithm}
Given the composite polynomials coefficients computed offline, we implement the run-time algorithm in Algorithm~\ref{alg:comp:run-time}. The first part computes the upper bound of $\normtwo{X}$ using Theorem~\ref{thm:comp:upper-bound}, and the second part iteratively applies the composite polynomial filter $\ftstar[T] \circ \ftstar[T-1] \circ \cdots \circ \ftstar[1]$ to $X$ after scaling. The approximation of the projection $\PiSnp{X}$ is then recovered by reconstructing $\frelu(X)$ from $\sign(X)$ and unscaling the result.






\begin{algorithm}[htbp]
    \caption{Run-time projection algorithm to compute~\eqref{eq:intro:psdcone-proj}}
    \label{alg:comp:run-time}
\begin{algorithmic}[1]
    \REQUIRE Symmetric matrix $X$, composite step $T$, polynomials $\{ \ftstar[t] \}_{t=1}^T$.
    \ENSURE Approximation of the projection of $X$ onto the PSD cone $\PiSnp{X}$.

    \STATE Compute the largest Ritz value pair $(\sigma, q)$ with a 20-step Lanczos run on $X^2$.

    \STATE Compute the upper bound of $\normtwo{X}$:
    \begin{equation}
        \tilde{\lambda} \leftarrow \sqrt{\sigma + \normtwo{X^2 q - \sigma q}}.
    \end{equation}

    \STATE Rescale $X$:
    \begin{equation}
        X_0 \leftarrow X / \tilde{\lambda}
    \end{equation}

    \FOR{$t = 1$ \TO $T$}
        \STATE Compute:
        \begin{equation}
            X_t \leftarrow \ftstar[t](X_{t-1}).
        \end{equation}
    \ENDFOR

    \RETURN $\tilde{\lambda} \cdot \frac{1}{2}X_0(I_n + X_T)$
\end{algorithmic}
\end{algorithm}

%% file: tables/single_precision.tex

\begin{table}[htbp]
    \centering
    \begin{adjustbox}{width=\linewidth}
    \begin{tabular}{ccc}
        \toprule
         & $\fstar_{\mathsf{single}}$ & $\ftstar_{\mathsf{single}}$ \\
        \hline 
        formula &
        $\displaystyle
        \begin{aligned}
            & \fstar[1](x) = 8.5098853026 x - 25.2643041908 x^3 + 18.7535678997 x^5 \\ 
            & \fstar[2](x) = 4.2495734789 x - \phantom{2}3.1549764881 x^3 + \phantom{1}0.5858847825 x^5 \\ 
            & \fstar[3](x) = 4.2251221908 x - \phantom{2}3.1380444351 x^3 + \phantom{1}0.5839534551 x^5 \\ 
            & \fstar[4](x) = 4.1248386870 x - \phantom{2}3.0683324528 x^3 + \phantom{1}0.5760029536 x^5 \\ 
            & \fstar[5](x) = 3.7580103358 x - \phantom{2}2.8092738924 x^3 + \phantom{1}0.5464842066 x^5 \\ 
            & \fstar[6](x) = 2.8561775413 x - \phantom{2}2.1340562332 x^3 + \phantom{1}0.4701107692 x^5 \\ 
            & \fstar[7](x) = 2.0206004158 x - \phantom{2}1.4037211505 x^3 + \phantom{1}0.3906738969 x^5 \\ 
            & \fstar[8](x) = 1.8758751005 x - \phantom{2}1.2509719905 x^3 + \phantom{1}0.3750972123 x^5 \\ 
            & \fstar[9](x) = 1.8750000000 x - \phantom{2}1.2500000000 x^3 + \phantom{1}0.3750000000 x^5 \\ 
            & \!\!\fstar[10](x) = 1.8750000000 x - \phantom{2}1.2500000000 x^3 + \phantom{1}0.3750000000 x^5 
        \end{aligned}$
        & 
        $\displaystyle
        \begin{aligned}
            & \ftstar[1](x) = 8.3119043343 x - 23.0739115930 x^3 + 16.4664144722 x^5 \\ 
            & \ftstar[2](x) = 4.1439360087 x - \phantom{2}2.9176674704 x^3 + \phantom{1}0.5246212487 x^5 \\ 
            & \ftstar[3](x) = 4.0257813209 x - \phantom{2}2.9025002398 x^3 + \phantom{1}0.5334261214 x^5 \\ 
            & \ftstar[4](x) = 3.5118574347 x - \phantom{2}2.5740236523 x^3 + \phantom{1}0.5050097282 x^5 \\ 
            & \ftstar[5](x) = 2.4398158400 x - \phantom{2}1.7586675341 x^3 + \phantom{1}0.4191290613 x^5 \\ 
            & \ftstar[6](x) = 1.9779835097 x - \phantom{2}1.3337358510 x^3 + \phantom{1}0.3772169049 x^5 \\ 
            & \ftstar[7](x) = 1.9559726949 x - \phantom{2}1.3091355170 x^3 + \phantom{1}0.3746734515 x^5 \\ 
            & \ftstar[8](x) = 1.9282822454 x - \phantom{2}1.2823649693 x^3 + \phantom{1}0.3704626545 x^5 \\ 
            & \ftstar[9](x) = 1.9220135179 x - \phantom{2}1.2812524618 x^3 + \phantom{1}0.3707011753 x^5 \\ 
            & \!\!\ftstar[10](x) = 1.8942192942 x - \phantom{2}1.2613293407 x^3 + \phantom{1}0.3676616051 x^5 
        \end{aligned}$
        \\    
        \hline 
        $e_\float$ & $1.1092 \times 10^{-5}$ & $8.7023 \times 10^{-6}$ \\
        \bottomrule
    \end{tabular}
    \end{adjustbox}
    \caption{Composite polynomials for single-precision PSD cone projection.} 
    \label{tab:comp:single}
\end{table}

%% file: tables/half_precision.tex

\begin{table}[htbp]
    \centering
    \begin{adjustbox}{width=\linewidth}
    \begin{tabular}{ccc}
        \toprule
         & $\fstar_{\mathsf{half}}$ & $\ftstar_{\mathsf{half}}$ \\
        \hline 
        formula &
        $\displaystyle
        \begin{aligned}
            & \fstar[1](x) = 8.4703288038 x - 25.1080747067 x^3 + 18.6292755991 x^5 \\ 
            & \fstar[2](x) = 4.1828341833 x - \phantom{2}3.1087011099 x^3 + \phantom{1}0.5806066814 x^5 \\ 
            & \fstar[3](x) = 3.9618572790 x - \phantom{2}2.9540637464 x^3 + \phantom{1}0.5629761180 x^5 \\ 
            & \fstar[4](x) = 3.2865862170 x - \phantom{2}2.4647201345 x^3 + \phantom{1}0.5073576939 x^5 \\ 
            & \fstar[5](x) = 2.2737499945 x - \phantom{2}1.6446603679 x^3 + \phantom{1}0.4161909275 x^5 \\ 
            & \fstar[6](x) = 1.8887161973 x - \phantom{2}1.2651572253 x^3 + \phantom{1}0.3765189256 x^5 \\ 
            & \fstar[7](x) = 1.8750008858 x - \phantom{2}1.2500009843 x^3 + \phantom{1}0.3750000984 x^5
        \end{aligned}$
        & 
        $\displaystyle
        \begin{aligned}
            & \ftstar[1](x) = 8.2885332412 x - 22.5927099246 x^3 + 15.8201383114 x^5 \\ 
            & \ftstar[2](x) = 4.1666196466 x - \phantom{2}2.9679004036 x^3 + \phantom{1}0.5307623217 x^5 \\ 
            & \ftstar[3](x) = 4.0611848147 x - \phantom{2}2.9698947955 x^3 + \phantom{1}0.5492133813 x^5 \\ 
            & \ftstar[4](x) = 3.6678301399 x - \phantom{2}2.7561018955 x^3 + \phantom{1}0.5421513305 x^5 \\ 
            & \ftstar[5](x) = 2.7632556383 x - \phantom{2}2.0607754898 x^3 + \phantom{1}0.4695405857 x^5 \\ 
            & \ftstar[6](x) = 2.0527445797 x - \phantom{2}1.4345145882 x^3 + \phantom{1}0.4070669182 x^5 \\ 
            & \ftstar[7](x) = 1.8804816691 x - \phantom{2}1.2583997294 x^3 + \phantom{1}0.3779501813 x^5
        \end{aligned}$
        \\    
        \hline 
        $e_\float$ & $7.2868 \times 10^{-5}$ & $4.9233 \times 10^{-5}$ \\
        \bottomrule
    \end{tabular}
    \end{adjustbox}
    \caption{Composite polynomials for half-precision PSD cone projection.}
    \label{tab:comp:half}
\end{table}

%% file: sections/evaluation.tex

\section{Implementation and Evaluation}
\label{sec:evaluation}

We implemented all algorithms in native C++ and CUDA, including the matrix-rescaling routine based on Lanczos iterations and the evaluation of composite matrix polynomials. The resulting codebase was benchmarked both on NVIDIA B200 GPUs ($192$ GB RAM) running CUDA~12.9 and on NVIDIA H100 GPUs ($80$ GB RAM) running CUDA~12.4. On the B200, we enabled the newest \texttt{BF16x9} emulation features introduced in CUDA~12.9 for the Blackwell architecture, yielding significantly faster single-precision general GEMMs~\cite{cublas129}. 
All Lanczos iterations are executed on the general CUDA cores in double precision (\doubleprec), mitigating the rapid loss of orthogonality. For composite matrix polynomials, we evaluate $\ftstar_{\mathsf{half}}$ (\cf Table~\ref{tab:comp:half}) on the Tensor Cores in half precision (\halfprec). In the case of $\ftstar_{\mathsf{single}}$ (\cf Table~\ref{tab:comp:single}), single precision (\singleprec) is employed on the H100's CUDA cores, whereas on the B200 we leverage Tensor Core \singleprec{} with emulation enabled. In terms of memory, our composite polynomial implementations require a workspace equivalent to the size of three additional matrices only. We use the state-of-the-art commercial CUDA library \cusolver{} for the \doubleprec{} and \singleprec{} eigenvalue decomposition on GPUs.

\subsection{Benchmarking PSD Cone Projection}
\label{sec:sec:eva:psdcone}
\paragraph{Datasets.}
We evaluate the performance of the proposed method on 33 dense datasets from the Matrix Depot package~\cite{zhang16peerj-matrixdepot}. Out of the 56 datasets available in the package, we removed 23 datasets of which the coefficients would overflow the double precision floating point representation for large matrix sizes, as well as datasets with constraints on the matrix size (e.g., the size being a power of 2). For each of the remaining 33 datasets, we generated 3 instances, of sizes $5000$, $10000$ and $20000$. The matrices were symmetrized by averaging the matrix with its transpose ($A_{\mathsf{sym}} = \frac{1}{2}(A + A\tran)$).


\paragraph{Benchmark methods and metrics.}
We restrict our benchmarks exclusively to GPUs, as CPU-based projection methods are not competitive for large-scale problems. For example, on a $10,000 \times 10,000$ dense symmetric matrix, MATLAB's \texttt{eig} function requires approximately $15$ seconds on a high-performance workstation equipped with a 64-Core Processor and $1$ TB RAM, whereas NVIDIA's \cusolver{} completes the eigenvalue decomposition in about $1$ second on both B200 and H100 GPUs. 
We systematically compare nine methods: 
\begin{itemize}
    \item \CSdouble: classical factorization-based method with \cusolver. We also use this method as a ground truth for the PSD projection in the error metric.
    \item \CSsingle: same as \CSdouble, single precision. 
    \item \PEsingle: composite coefficients proposed in~\cite{amsel25arxiv-polar-express} and implemented in \singleprec. We set $\ell = 10^{-3}$ and $T = 10$, resulting in $31$ GEMMs.
    \item \PEhalf: composite coefficients proposed in~\cite{amsel25arxiv-polar-express} and implemented in \halfprec. We set $\ell = 10^{-3}$ and $T = 7$, resulting in $22$ GEMMs.
    \item \NSsingle: $15$ Newton-Schulz iterations~\cite{francisco17laa-fixedpoint-approximate-psdcone} implemented in \singleprec, resulting in $31$ GEMMs.
    \item \NShalf: $10$ Newton-Schulz iterations~\cite{francisco17laa-fixedpoint-approximate-psdcone} implemented in \halfprec, resulting in $21$ GEMMs.
    \item \CPsingle: the proposed method ($\ftstar_{\mathsf{single}}$) in \singleprec{} with $31$ GEMMs. 
    \item \CPsingle{} (em.): same as \CPsingle, with \texttt{BF16x9} emulation enabled on B200 GPU.
    \item \CPhalf: the proposed method ($\ftstar_{\mathsf{half}}$) in \halfprec{} with $22$ GEMMs.
\end{itemize}
We evaluate the performance of the methods using two metrics: \emph{time} and \emph{relative error}. The time includes both Lanczos iterations, data type conversion, and composite polynomial filtering. Given the input matrix $A \in \Sn$, denote an algorithm's output matrix as $A_{+}$. The relative error is defined as 
\[
    \frac{
        \normF{A_{+} - \PiSnp{A}}
    }{
        \normF{\PiSnp{A}}
    },
\]
where $\PiSnp{A}$ is from \CSdouble. The relative error is computed in \doubleprec, regardless of the data type used for the PSD projection. For the same matrix, both metrics are averaged over 5 runs. The code used for the benchmarks is available at:

\vspace{2mm}
\begin{center}
    \url{https://github.com/ComputationalRobotics/psd_projection_benchmarks}.
\end{center}
\vspace{2mm}


\paragraph{Results.}
A synthetic view of the evaluation results on the B200 GPU is presented in Figure~\ref{fig:psd_time_boxplot} and~\ref{fig:psd_error_boxplot}, with complementary statistics summarized in Tables~\ref{tab:benchmark_stats_5000_B200}, \ref{tab:benchmark_stats_10000_B200}, and \ref{tab:benchmark_stats_20000_B200}. Detailed, dataset-specific results for both B200 and H100 GPUs are provided in \autoref{sec:appendix:complete_comparison}.
In terms of relative error, our proposed composite methods consistently outperform the other factorization-free approaches of their respective precisions, namely Polar Express~\cite{amsel25arxiv-polar-express} and Newton-Schulz~\cite{francisco17laa-fixedpoint-approximate-psdcone}. In \halfprec{}, \CPhalf{} reliably achieves $10^{-3}$ relative error across a wide range of datasets while offering more than an order-of-magnitude speed-up over both \CSdouble{} and \CSsingle{}. Even for large dense matrices with $n = 20000$, it is able to finish projection within $300$ms. In \singleprec{}, although \CSsingle{} generally attains the highest accuracy, \CPsingle{} (with emulation) can achieve speed-ups of up to $2\times$. 
As already noticeable from Figure~\ref{fig:psd_error_boxplot}, there are several datasets (e.g., the \texttt{triw} dataset) on which all composite methods perform poorly. A common characteristic of these matrices is the presence of one dominant extremal eigenvalue accompanied by many small but nonzero eigenvalues. As recommended in~\cite{parlett1987book-symmetric-eig}, applying a deflation technique to first remove the extremal eigenpairs could significantly improve performance in such cases.

\begin{remark}
    We emphasize that the composite baseline methods from Polar Express and Newton–Schulz implemented here are already considerably improved compared to their original formulations~\cite{amsel25arxiv-polar-express,francisco17laa-fixedpoint-approximate-psdcone}, due to our tighter spectral norm outer approximations described in \S~\ref{sec:comp:finite-prec}.
\end{remark}

\input{figs/B200_boxplots.tex}

\input{figs/psd_projection_tables/B200_stats_5000.tex}
\input{figs/psd_projection_tables/B200_stats_10000.tex}
\input{figs/psd_projection_tables/B200_stats_20000.tex}

\subsection{Integration into Semidefinite Programming}

\paragraph{ADMM for SDP.}
We further assess our low-precision composite PSD-cone projection by embedding it in a first-order solver for the semidefinite programming:
\begin{equation}
    \label{eq:exp:sdp}
    \begin{array}{llllll}
        \text{Primal:} & \text{minimize}   & \inprod{C}{X} 
                      & \qquad\qquad
        \text{Dual:}   & \text{maximize}   & b\tran y \\[2pt]
                      & \text{subject to} & \Asdp X = b 
                      &                   & \text{subject to} & \AsdpT y + S = C \\[2pt]
                      &                   & X \in \Symp{n}    
                      &                   &                   & S \in \Symp{n},
    \end{array}
\end{equation}
where the linear map $\Asdp : \Sn \to \Real{m}$ is defined by $\Asdp X := (\inprod{A_1}{X},\dots,\inprod{A_m}{X})$ and its adjoint $\AsdpT y = \sum_{i=1}^m y_i A_i$. We assume the matrices $\{A_i\}_{i=1}^m$ are linearly independent. Problem~\eqref{eq:exp:sdp} is solved with the three-step Alternating Direction Method of Multipliers (ADMM) scheme
from~\cite{wen10mpc-admmsdp,kang25arxiv-admm}:
\begin{subequations}
    \label{eq:exp:admm-three-step}
    \begin{align}
        \ykpo &= (\Asdp \AsdpT)^{-1}
                \!\left(\sigma^{-1} b - \Asdp\bigl(\sigma^{-1}\Xk + \Sk - C\bigr)\right),\\[2pt]
        \Skpo &= \Pi_{\Symp{n}}\!\left(C - \AsdpT \ykpo - \sigma^{-1}\Xk\right),\\[2pt]
        \Xkpo &= \Xk + \sigma\!\left(\Skpo + \AsdpT \ykpo - C\right),
    \end{align}
\end{subequations}
with penalty parameter $\sigma > 0$. For large-scale SDPs, the PSD-cone projection $\Pi_{\Symp{n}}(\cdot)$ is typically the main computational bottleneck. We therefore replace $\Pi_{\Symp{n}}(\cdot)$ with our composite polynomial filtering, implemented in both \singleprec{} and \halfprec{} arithmetic. 

\paragraph{Experiment set-up.} We choose four large-scale instances from the Mittelmann SDP test set~\cite{mittelmann2006dataset-sparse-sdp-problems} whose matrix dimension is at least $5000$:
\texttt{G55mc} ($n = 5000$), \texttt{G59mc} ($n = 5000$), \texttt{G60\_mb} ($n = 7000$) and \texttt{G60mc} ($n = 7000$).
The ADMM iteration limit is fixed at $5000$, and the algorithm stops once the maximum KKT residual $\eta$ falls below $10^{-4}$, where
\[
\resizebox{\textwidth}{!}{$
    \eta := \max \left\{
        \frac{\normtwo{\Asdp X - b}}{1 + \normtwo{b}},\;
        \frac{\normF{\AsdpT y + S - C}}{1 + \normF{C}},\;
        \frac{\abs{\inprod{C}{X} - b\tran y}}{1 + \abs{\inprod{C}{X}} + \abs{b\tran y}},\;
        \frac{\min \{ 0, \lambda_{\min}(X) \}}{1 + \normtwo{b}},\;
        \frac{\min \{ 0, \lambda_{\min}(S) \}}{1 + \normF{C}}
    \right\}.
$}
\]
The ADMM iterate~\eqref{eq:exp:admm-three-step} is warm-started with our composite polynomial filters.  
Because no eigenvalues are computed explicitly in this phase, we monitor only the quantity 
\[
\max \left\{
    \frac{\normtwo{\Asdp X - b}}{1 + \normtwo{b}}, 
    \frac{\normF{\AsdpT y + S - C}}{1 + \normF{C}},  
    \frac{\abs{\inprod{C}{X} - b\tran y}}{1 + \abs{\inprod{C}{X}} + \abs{b\tran y}}
 \right\},
\] 
using it as a surrogate for $\eta$.  Once this surrogate first drops below $10^{-2}$, we switch to the \doubleprec{} PSD-cone projection based on an eigenvalue decomposition.

\input{figs/ADMM.tex}

\paragraph{Results.} The results are shown in Figure~\ref{figs:exp:admm}. 
In every experiment we examine, \emph{a posteriori}, the minimum eigenvalues of the final iterates $\Xk$ and $\Sk$; they remain far below $10^{-8}$---even for \texttt{G60mc} in \singleprec{}, where the algorithm never switches to the high-precision projection. This confirms that our surrogate for $\eta$ is adequate.
For the two easier instances (\texttt{G55mc} and \texttt{G59mc}) ADMM reduces $\eta$ below $10^{-4}$ within $5000$ iterations. Across \doubleprec{}, \singleprec{}, and \halfprec{} arithmetic the evolution of $\eta$ is virtually identical, showing that warming up the first 100 iterations with composite filtering sacrifices almost no accuracy.
The harder instances (\texttt{G60mc} and \texttt{G60\_mb}) resist pushing $\eta$ below $10^{-2}$.  
Nevertheless, the \singleprec{} and \halfprec{} projections---although noisier---do not slow convergence.  
In three of the four problems, low-precision projections actually deliver a smaller final $\eta$ while reducing the total projection time by up to an order of magnitude. These findings demonstrate that our low-precision projection strategy is both effective and efficient for large-scale semidefinite optimization.

\begin{remark}
    One might question the overall benefit of warm-starting, given that we eventually revert to high-precision projections to obtain higher-accuracy SDP solutions. However, empirical evidence shows that ADMM quickly enters a \emph{low-rank} regime in which either $\Xk$ or $\Sk$ maintains low rank~\cite{souto2022opt-low-rank-sdp-approximate-operator-splitting,kang25arxiv-admm}. In that phase, iterative low-rank eigensolvers such as LOBPCG~\cite{knyazev2001sisc-lobpcg} can perform the PSD projection in nearly $\mathcal{O}(n^{2})$ time per iteration.  Indeed, an \emph{a posteriori} inspection of the final $\Xk$ ranks confirms this behavior:  
    (a) \texttt{G55mc}: $53$;  
    (b) \texttt{G59mc}: $21$;  
    (c) \texttt{G60\_mb}: $224$ (\doubleprec), $248$ (\singleprec), $178$ (\halfprec);  
    (d) \texttt{G60mc}: $119$ (\doubleprec), $298$ (\singleprec), $161$ (\halfprec).  
    Thus, our low-precision projection methods could significantly accelerate entry into this advantageous \emph{low-rank} regime.
\end{remark}

%% file: figs/B200_boxplots.tex
\begin{figure}[ht]
    \centering
    \vspace{-2mm}
    \includegraphics[width=\textwidth]{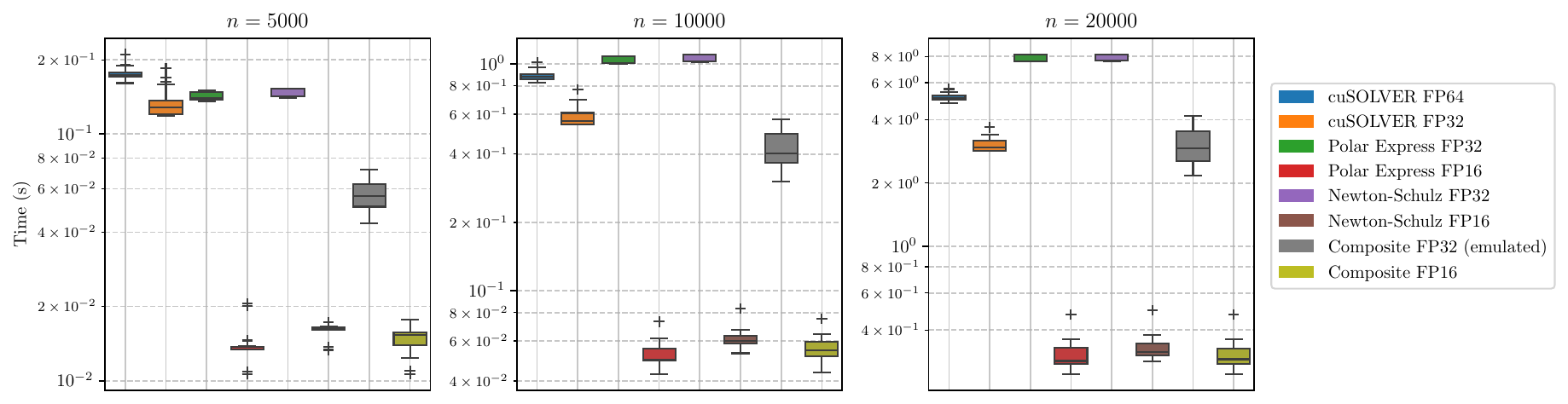}

    \caption{Boxplots for different PSD cone projection methods' \emph{execution time} on B200 GPU. \label{fig:psd_time_boxplot}}
    
    
\end{figure}

\begin{figure}[h]
    \centering
    \vspace{-8mm}
    \includegraphics[width=\textwidth]{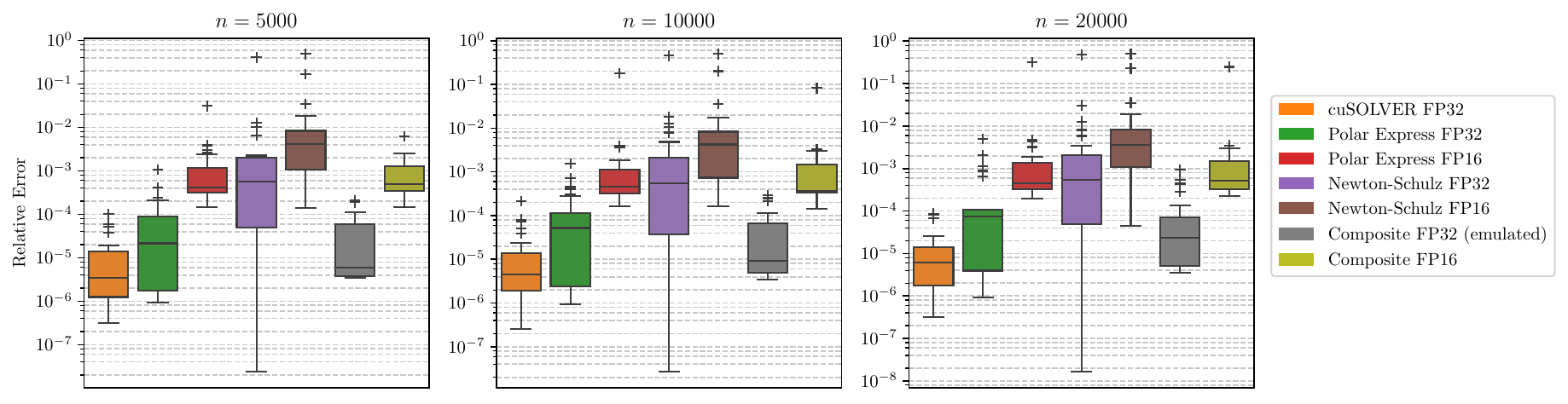}

    \caption{Boxplots for different PSD cone projection methods' \emph{relative error} on B200 GPU. \label{fig:psd_error_boxplot}}

    
\end{figure}

%% file: figs/psd_projection_tables/B200_stats_5000.tex
\begin{table}[ht]
    \begin{adjustbox}{width=\linewidth}
    \begin{tabular}{lllllllll}
        \toprule
        {} & \multicolumn{3}{c}{Relative Error} & \multicolumn{3}{c}{Time (s)} \\
        {} & {\quad Mean} & {\quad Median} & {\quad Std.} & {\quad Mean} & {\quad Median} & {\quad Std.} \\
        \midrule
        \CSdouble & \num{0.00e+00} & \num{0.00e+00} & \num{0.00e+00} & \num{1.75e-01} & \num{1.74e-01} & \num{1.03e-02} \\
        \CSsingle & \num{1.25e-05} & \num{3.40e-06} & \num{2.18e-05} & \num{1.32e-01} & \num{1.28e-01} & \num{1.63e-02} \\
        \PEsingle & \num{9.16e-05} & \num{2.14e-05} & \num{1.98e-04} & \num{1.43e-01} & \num{1.40e-01} & \num{5.47e-03} \\
        \PEhalf & \num{1.83e-03} & \num{4.19e-04} & \num{5.36e-03} & \num{1.37e-02} & \num{1.37e-02} & \num{1.99e-03} \\
        \NSsingle & \num{1.47e-02} & \num{5.65e-04} & \num{7.07e-02} & \num{1.47e-01} & \num{1.43e-01} & \num{5.56e-03} \\
        \NShalf & \num{2.56e-02} & \num{4.07e-03} & \num{8.79e-02} & \num{1.60e-02} & \num{1.62e-02} & \num{1.00e-03} \\
        \CPsingle{} (em.) & \num{3.71e-05} & \num{5.96e-06} & \num{5.28e-05} & \num{5.59e-02} & \num{5.59e-02} & \num{7.36e-03} \\
        \CPhalf & \num{9.53e-04} & \num{4.86e-04} & \num{1.16e-03} & \num{1.48e-02} & \num{1.53e-02} & \num{1.67e-03} \\
        \bottomrule
    \end{tabular}
    \end{adjustbox}
    \caption{Statistics of the PSD projection methods on datasets of size 5000 for B200 GPU. \label{tab:benchmark_stats_5000_B200}}
\end{table}

%% file: figs/psd_projection_tables/B200_stats_10000.tex
\begin{table}[!ht]
    \begin{adjustbox}{width=\linewidth}
    \begin{tabular}{lllllllll}
        \toprule
        {} & \multicolumn{3}{c}{Relative Error} & \multicolumn{3}{c}{Time (s)} \\
        {} & {\quad Mean} & {\quad Median} & {\quad Std.} & {\quad Mean} & {\quad Median} & {\quad Std.} \\
        \midrule
        \CSdouble & \num{0.00e+00} & \num{0.00e+00} & \num{0.00e+00} & \num{8.85e-01} & \num{8.77e-01} & \num{4.47e-02} \\
        \CSsingle & \num{1.89e-05} & \num{4.48e-06} & \num{4.06e-05} & \num{5.84e-01} & \num{5.58e-01} & \num{5.36e-02} \\
        \PEsingle & \num{1.57e-04} & \num{5.21e-05} & \num{3.01e-04} & \num{1.04e+00} & \num{1.01e+00} & \num{3.75e-02} \\
        \PEhalf & \num{6.38e-03} & \num{4.51e-04} & \num{3.08e-02} & \num{5.17e-02} & \num{4.97e-02} & \num{5.90e-03} \\
        \NSsingle & \num{1.61e-02} & \num{5.55e-04} & \num{7.87e-02} & \num{1.06e+00} & \num{1.03e+00} & \num{3.79e-02} \\
        \NShalf & \num{2.65e-02} & \num{4.22e-03} & \num{9.08e-02} & \num{6.09e-02} & \num{5.98e-02} & \num{5.33e-03} \\
        \CPsingle{} (em.) & \num{4.75e-05} & \num{9.12e-06} & \num{7.37e-05} & \num{4.16e-01} & \num{4.03e-01} & \num{7.42e-02} \\
        \CPhalf & \num{3.45e-03} & \num{3.70e-04} & \num{1.42e-02} & \num{5.48e-02} & \num{5.45e-02} & \num{6.41e-03} \\
        \bottomrule
    \end{tabular}
    \end{adjustbox}
    \caption{Statistics of the PSD projection methods on datasets of size 10000 for B200 GPU. \label{tab:benchmark_stats_10000_B200}}
\end{table}

%% file: figs/psd_projection_tables/B200_stats_20000.tex
\begin{table}[!ht]
    \begin{adjustbox}{width=\linewidth}
    \begin{tabular}{lllllllll}
        \toprule
        {} & \multicolumn{3}{c}{Relative Error} & \multicolumn{3}{c}{Time (s)} \\
        {} & {\quad Mean} & {\quad Median} & {\quad Std.} & {\quad Mean} & {\quad Median} & {\quad Std.} \\
        \midrule
        \CSdouble & \num{0.00e+00} & \num{0.00e+00} & \num{0.00e+00} & \num{5.13e+00} & \num{5.08e+00} & \num{2.43e-01} \\
        \CSsingle & \num{1.73e-05} & \num{6.21e-06} & \num{2.69e-05} & \num{3.04e+00} & \num{2.96e+00} & \num{2.02e-01} \\
        \PEsingle & \num{4.08e-04} & \num{7.45e-05} & \num{9.46e-04} & \num{7.84e+00} & \num{7.58e+00} & \num{2.84e-01} \\
        \PEhalf & \num{1.09e-02} & \num{4.61e-04} & \num{5.60e-02} & \num{2.98e-01} & \num{2.85e-01} & \num{4.41e-02} \\
        \NSsingle & \num{1.69e-02} & \num{5.35e-04} & \num{8.28e-02} & \num{7.91e+00} & \num{7.64e+00} & \num{2.85e-01} \\
        \NShalf & \num{2.72e-02} & \num{3.69e-03} & \num{9.30e-02} & \num{3.25e-01} & \num{3.14e-01} & \num{3.92e-02} \\
        \CPsingle{} (em.) & \num{1.21e-04} & \num{2.36e-05} & \num{2.16e-04} & \num{2.99e+00} & \num{2.93e+00} & \num{5.58e-01} \\
        \CPhalf & \num{8.59e-03} & \num{5.14e-04} & \num{4.30e-02} & \num{3.01e-01} & \num{2.91e-01} & \num{4.42e-02} \\
        \bottomrule
    \end{tabular}
    \end{adjustbox}
    \caption{Statistics of the PSD projection methods on datasets of size 20000 for B200 GPU. \label{tab:benchmark_stats_20000_B200}}
\end{table}

%% file: figs/ADMM.tex

\begin{figure}[t]

    \begin{minipage}{\textwidth}
        \begin{tabular}{cccc}
            \begin{minipage}{0.22\textwidth}
                \centering
                \hspace{-4mm}\includegraphics[width=\columnwidth]{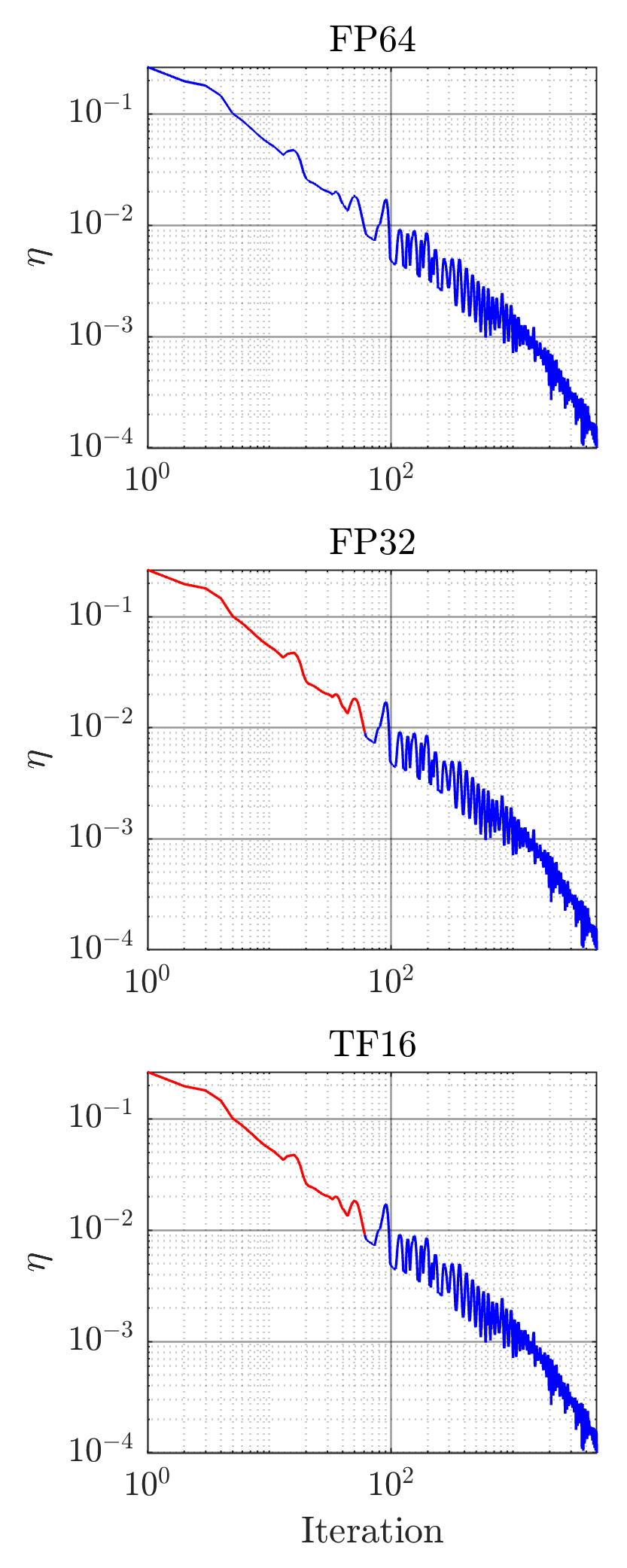}
                \makebox[\columnwidth][l]{(a) \texttt{G55mc}, $n = 5000$}
            \end{minipage}

            \begin{minipage}{0.22\textwidth}
                \centering
                \hspace{-4mm}\includegraphics[width=\columnwidth]{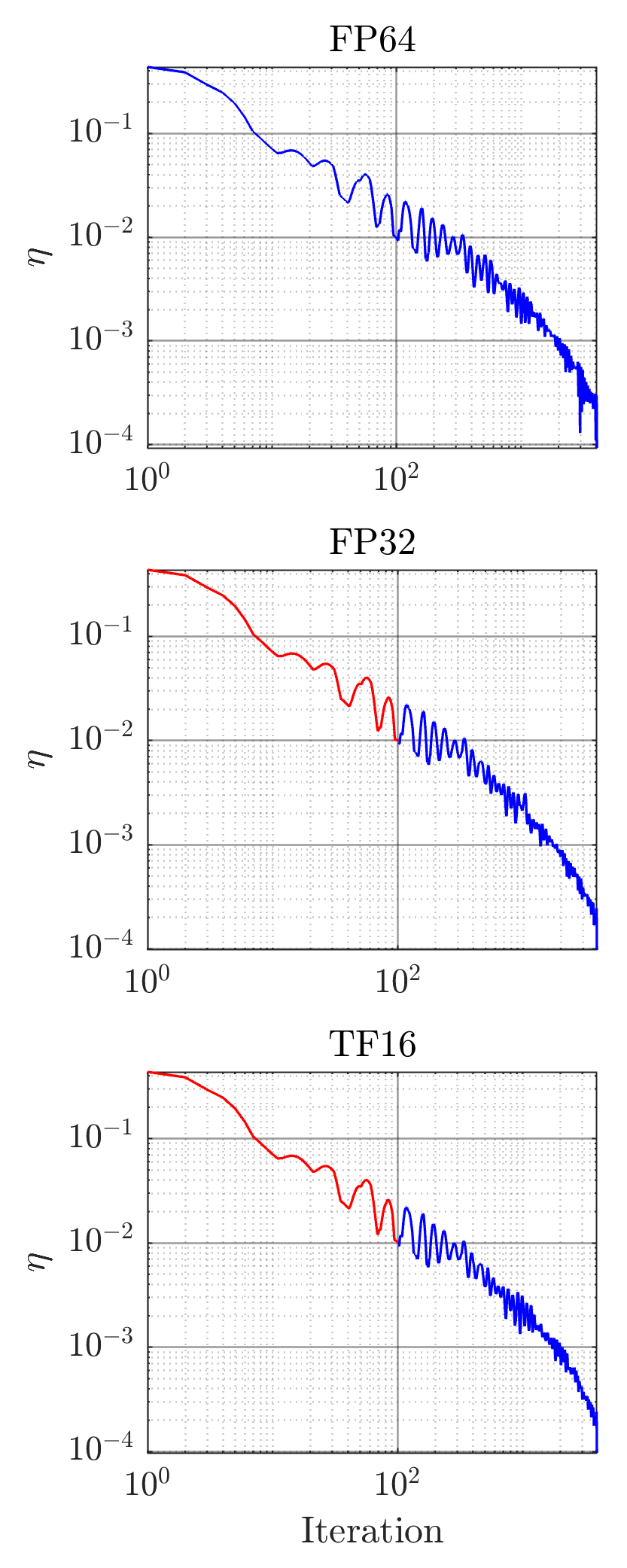}
                \makebox[\columnwidth][l]{(b) \texttt{G59mc}, $n = 5000$}
            \end{minipage}

            \begin{minipage}{0.22\textwidth}
                \centering
                \hspace{-4mm}\includegraphics[width=\columnwidth]{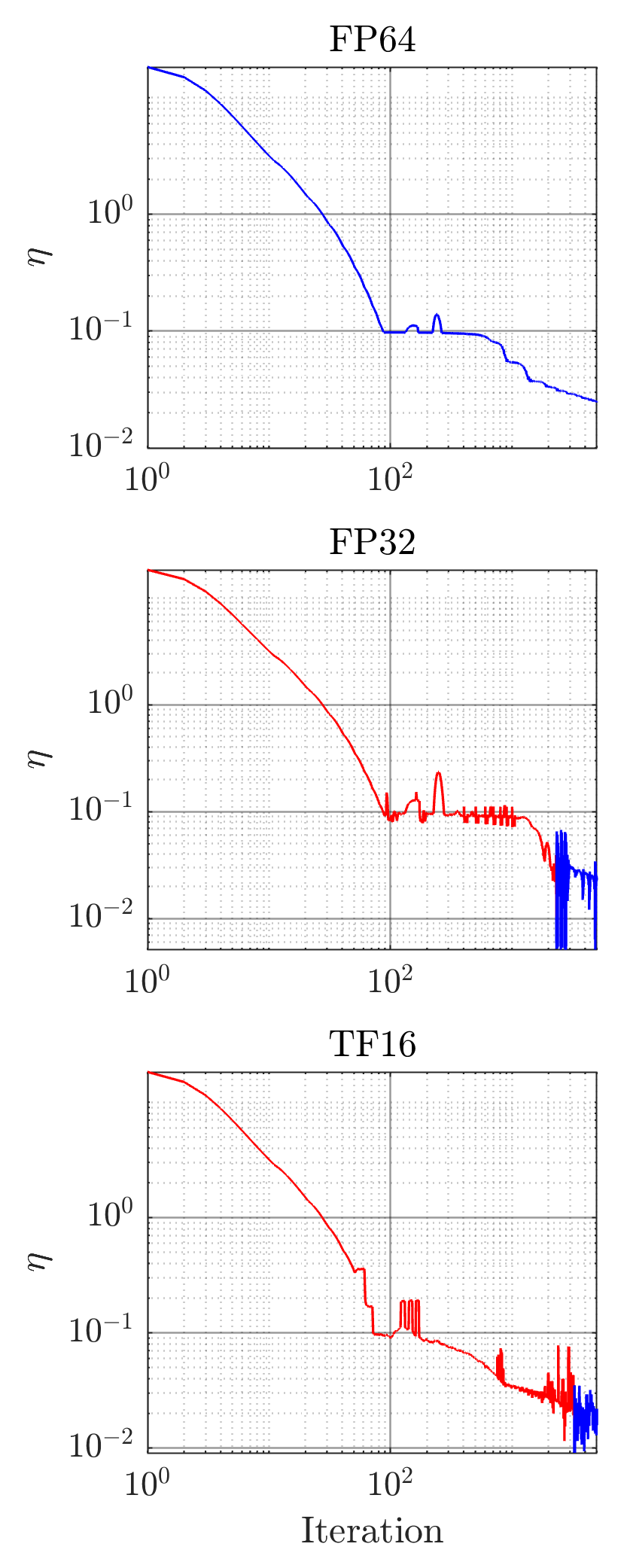}
                \makebox[\columnwidth][l]{(c) \texttt{G60\_mb}, $n = 7000$}
            \end{minipage}

            \begin{minipage}{0.22\textwidth}
                \centering
                \hspace{-4mm}\includegraphics[width=\columnwidth]{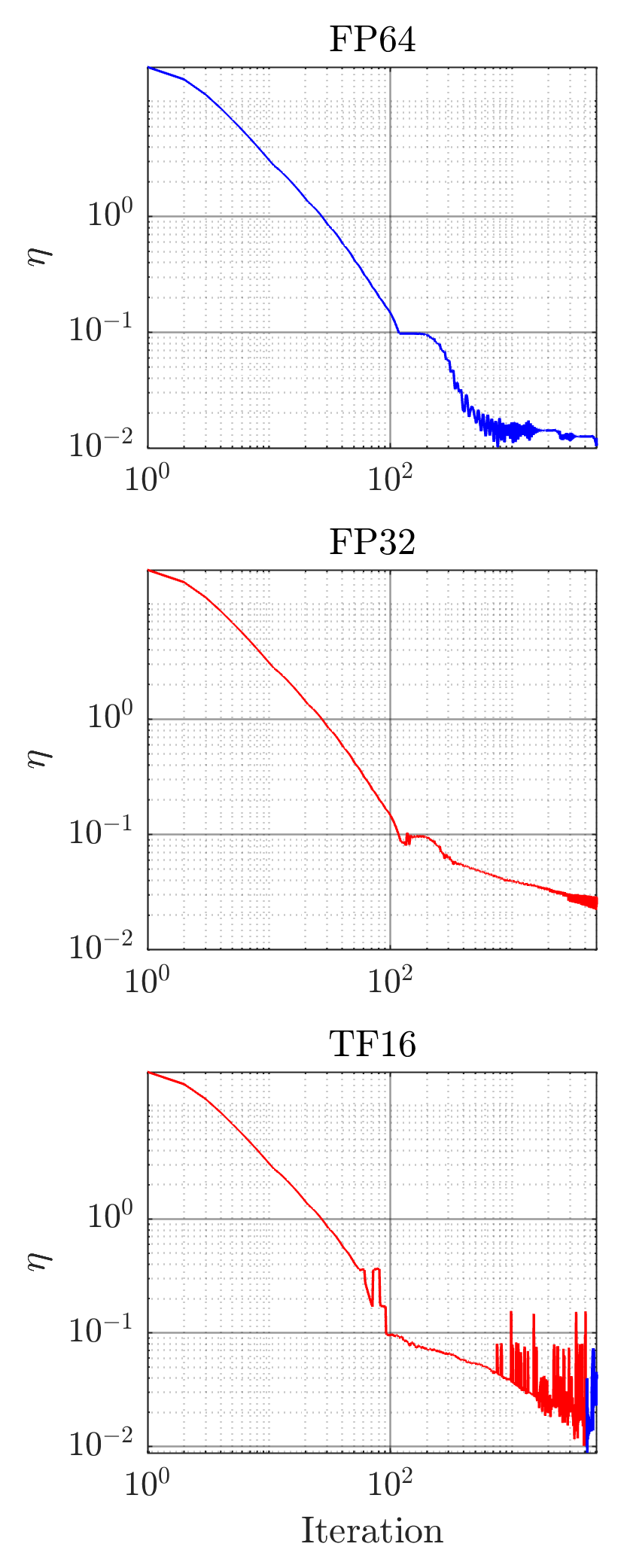}
                \makebox[\columnwidth][l]{(d) \texttt{G60mc}, $n = 7000$}
            \end{minipage}
        \end{tabular}
    \end{minipage}

    \caption{Numerical results for warm-starting ADMM with low-precision PSD-cone projection via composite polynomial filtering. \textcolor{red}{Red segments}: low-precision projection phase; \textcolor{blue}{blue segments}: high-precision phase.  
    Curves (top to bottom):  
    (1) baseline with purely \doubleprec{} projections;  
    (2) warm-start using \singleprec{} composite filtering;  
    (3) warm-start using \halfprec{} composite filtering.  \label{figs:exp:admm}}
\end{figure}

%% file: sections/conclusion.tex

\section{Conclusion}
\label{sec:conclusion}

We introduced a factorization-free PSD cone projection method that replaces expensive eigenvalue decompositions with a short cascade of low-degree polynomials executed entirely as GEMMs, achieving up to an order-of-magnitude acceleration on NVIDIA B200/H100 GPUs for moderate-accuracy projections. Used as a warm-start within a first-order SDP solver, the filter significantly reduces the total projection time without slowing convergence.

\paragraph{Limitations and future work.}
Composite filtering excels on dense spectra but can stumble when a few dominant eigenvalues are present; a preliminary deflation step may therefore be required. Future work will explore (\emph{i}) hybrid variants that couple the polynomial filter with low-rank corrections to cut GEMM counts further, and (\emph{ii}) adaptive strategies that adjust the polynomial depth online based on real-time spectral estimates.

\section*{Acknowledgment}

We thank Steven Dalton, Cris Cecka, Samuel Rodriguez Bernabeu, and Mikail Khona at NVIDIA for valuable discussions on numerical linear algebra methods on GPUs. Part of this research was conducted during visits by SK and HY to the Polynomial Optimization team at LAAS-CNRS, where Jean-Bernard Lasserre, Didier Henrion, and Victor Magron provided helpful insights.

%% file: sections/appendix.tex

\section{Remez Algorithm}
\label{app:remez}

This section provides Remez algorithm's necessary backgrounds for completeness. 
\begin{definition}[Minimax approximate polynomial]
    Given a continuous function $g$ over a compact interval $[a, b]$ ($a,b > 0$), $\fstar \in \mathbb{R}^{\odd}_d[x]$ is called a \textit{minimax approximate polynomial}, if and only if 
    \begin{align}
        \label{app:eq:minimax}
        \fstar = \argmin_{f \in \mathbb{R}^{\odd}_d[x]} \max_{x \in [a, b]} \abs{ f(x) - g(x) }
    \end{align}
    The existence and uniqueness of such an $\fstar$ is guaranteed~\cite[Lemma A.1]{amsel25arxiv-polar-express}.
\end{definition}

Remez algorithm~\cite{remez62ssr-remez} provides a systematic way to compute such an $\fstar$, as illustrated in Algorithm~\ref{app:alg:remez}.

\begin{algorithm}[htbp]
    \caption{Remez algorithm~\cite[Algorithm 1]{lee22tdsc-minimax-approximation-sign}}
    \label{app:alg:remez}
\begin{algorithmic}[1]
    \REQUIRE The target function $g$, compact interval $[a, b]$, a set of polynomial basis $\Phi := \{ \varphi_\ell \}_{\ell=1}^d$
    \ENSURE The minimax approximate function $\fstar$

    \STATE Initialization of extremal points: $d + 1$ Chebyshev nodes $\{ x_i \}_{i=1}^{d+1}$ on $[a, b]$

    \FOR{$k = 1$ to $\infty$}
        \STATE (1) Build and solve the alternation system. The linear system has $d+1$ variables: $\{ c_\ell \}_{\ell=1}^{d}$ and $E$. The $d+1$ linear equations are:
        \begin{align*}
            \sum_{\ell=1}^d c_\ell \varphi_\ell (x_i) - g(x_i) = (-1)^i E, \quad \forall i = 1, \dots , d+1.
        \end{align*}

        \STATE (2) Get the $d$ roots $\{ z_\ell \}_{\ell=1}^d$ of the residual $e(x) := \sum_{\ell=1}^d c_\ell \varphi_\ell (x) - g(x)$: 
        \begin{align*}
            x_1 < z_1 <x_2 < z_2 < x_3 < \dots < x_{d-1} < z_{d-1} < x_{d} < z_d < x_{d+1}
        \end{align*}

        \STATE (3) Update the $d+1$ extremal points. Divide $[a, b]$ into $d+1$ sub-intervals:
        \begin{align*}
            [z_0:=a, z_1], \ [z_1, z_2], \ \dots, \, [z_{d-1}, z_d], \ [z_d, z_{d+1}:=b]
        \end{align*}
        For each sub-interval $[z_{i-1}, z_i]$ ($i = 1, \dots, d+1$), find the maximum (resp. minimum) point of $e(x)$ if $e(x_i)$ is positive (resp. negative). Denote these $d+1$ extremal points as $\{ y_i \}_{i=1}^{d+1}$.

        \STATE \vspace{2mm}

        \STATE (4) Check convergence. If $\{ y_i \}_{i=1}^{d+1}$ is sufficiently close to $\{ x_i \}_{i=1}^{d+1}$, break. Else, replace $\{ x_i \}_{i=1}^{d+1}$ with $\{ y_i \}_{i=1}^{d+1}$.
    \ENDFOR
\end{algorithmic}
\end{algorithm}

\section{Complete Comparison of PSD Projection Methods}
\label{sec:appendix:complete_comparison}
This section provides the complete results of the comparison of the different PSD projection methods, including the relative error and execution time of each method on each dataset for both GPU architectures (H100 and B200). Tables~\ref{tab:benchmark_stats_5000_H100}, \ref{tab:benchmark_stats_10000_H100}, and \ref{tab:benchmark_stats_20000_H100} show the results for the H100 GPU aggregated over the datasets. Figures~\ref{fig:psd_time_boxplot_H100} and \ref{fig:psd_error_boxplot_H100} show the same results in boxplot format. Figures~\ref{fig:psd_b200_time} and \ref{fig:psd_h100_time} show the execution time of each method on the B200 and H100 GPUs, separating the different datasets. Figures~\ref{fig:psd_b200_error} and \ref{fig:psd_h100_error} show the relative error of each method on the B200 and H100 GPUs per dataset.

\input{figs/H100_boxplots.tex}
\input{figs/psd_projection_tables/H100_stats_5000.tex}
\input{figs/psd_projection_tables/H100_stats_10000.tex}
\input{figs/psd_projection_tables/H100_stats_20000.tex}

\newpage

\newgeometry{bottom=0cm,top=2cm,left=2cm,right=2cm} 
\includepdf[
    pages=1,    
    width=\paperheight,
    height=\paperwidth,
    angle=90,
    pagecommand={\thispagestyle{empty}\vspace*{\fill}\captionof{figure}{PSD projection time on B200 GPU. \label{fig:psd_b200_time}}}
]{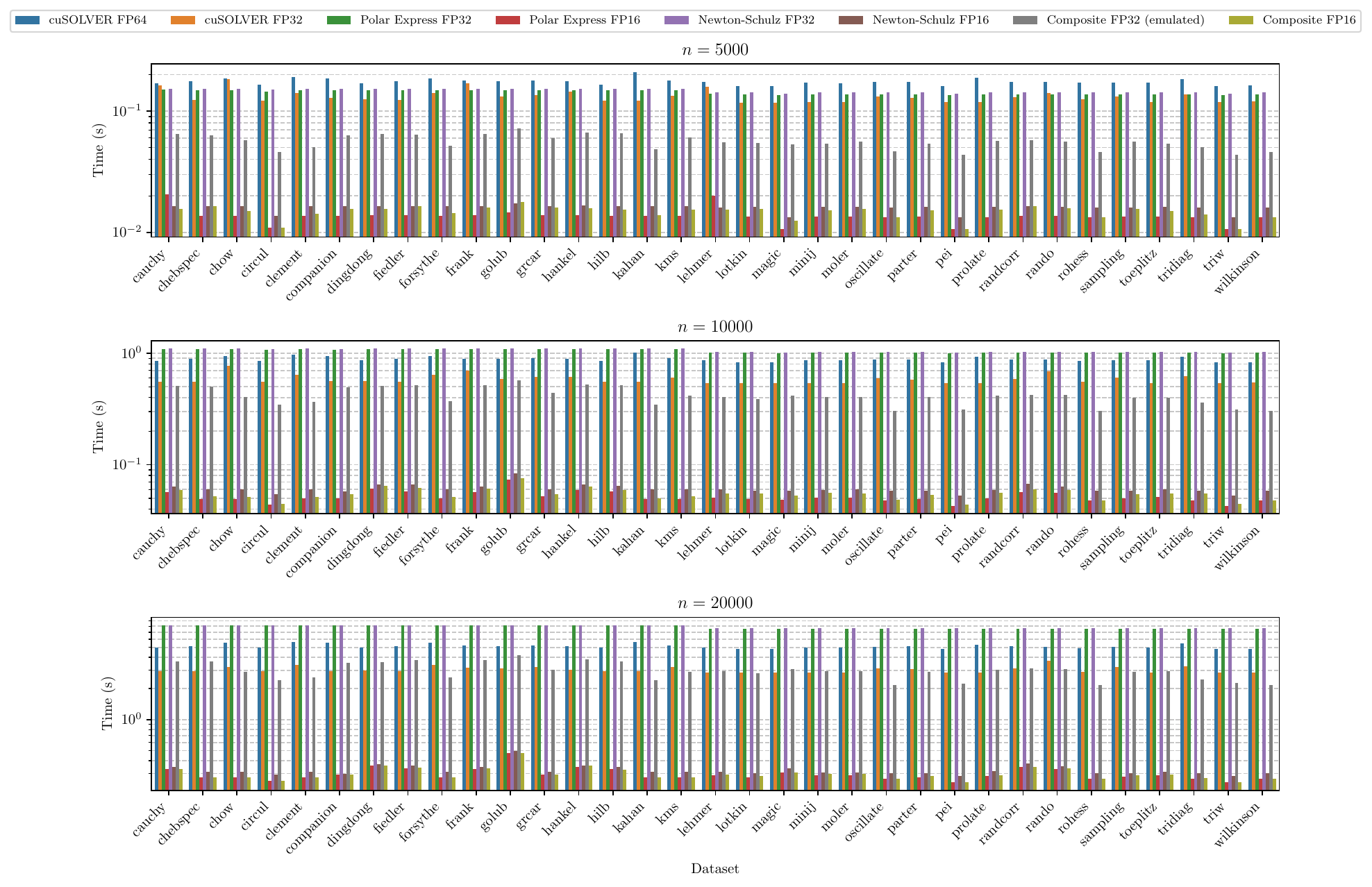}
\includepdf[
    pages=1,
    width=\paperheight,
    height=\paperwidth,
    angle=90,
    pagecommand={\thispagestyle{empty}\vspace*{\fill}\captionof{figure}{PSD projection time on H100 GPU. \label{fig:psd_h100_time}}}
]{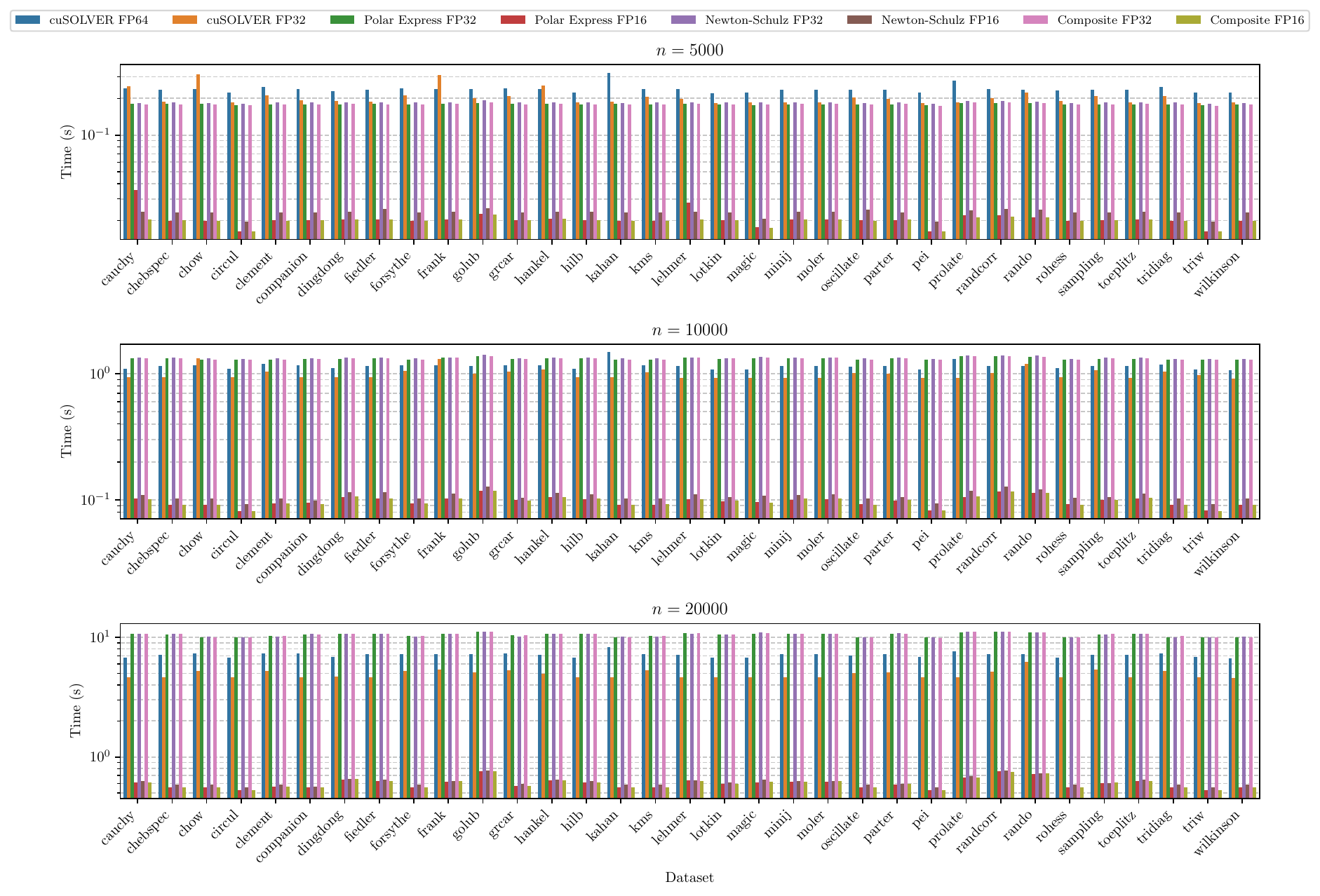}
\includepdf[
    pages=1,
    scale=0.96,
    offset=0 .5cm,
    angle=90,
    pagecommand={\thispagestyle{empty}\vspace*{\fill}\captionof{figure}{PSD projection error on B200 GPU. \label{fig:psd_b200_error}}}
]{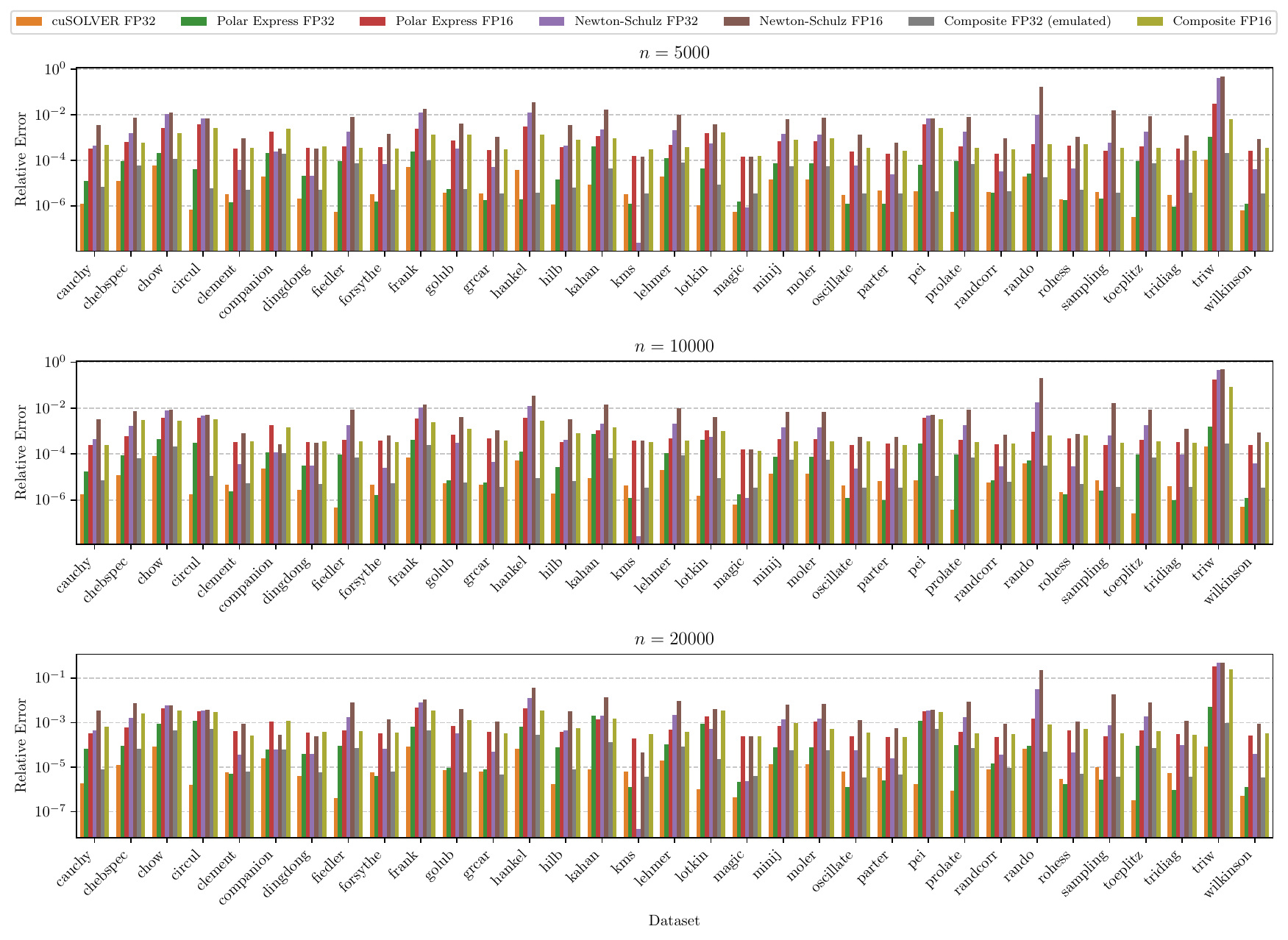}
\includepdf[
    pages=1,
    scale=0.96,
    offset=0 .5cm,
    angle=90,
    pagecommand={\thispagestyle{empty}\vspace*{\fill}\captionof{figure}{PSD projection error on H100 GPU. \label{fig:psd_h100_error}}}
]{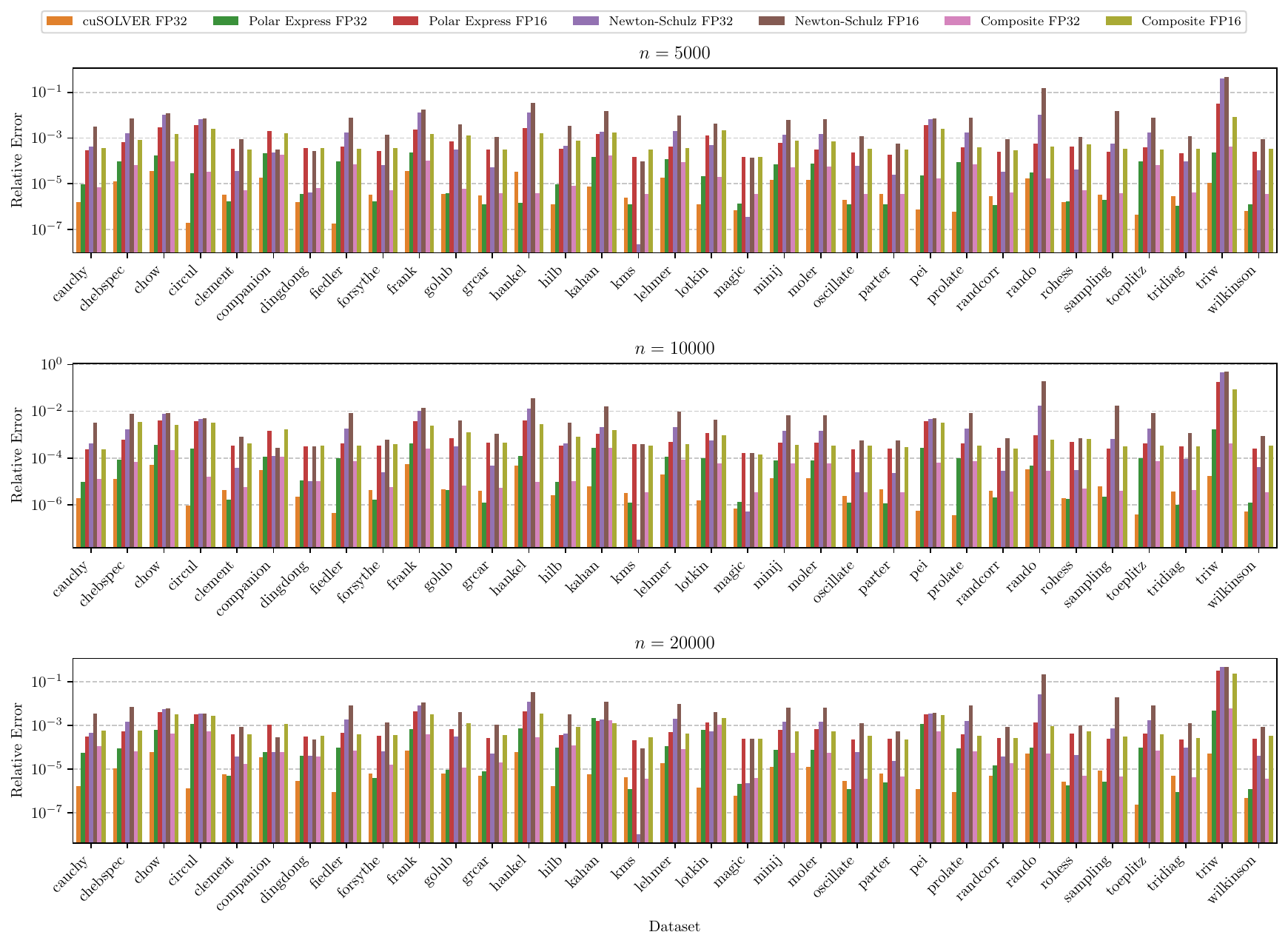}
\restoregeometry

%% file: figs/H100_boxplots.tex
\begin{figure}[ht]
    \centering
    \includegraphics[width=\textwidth]{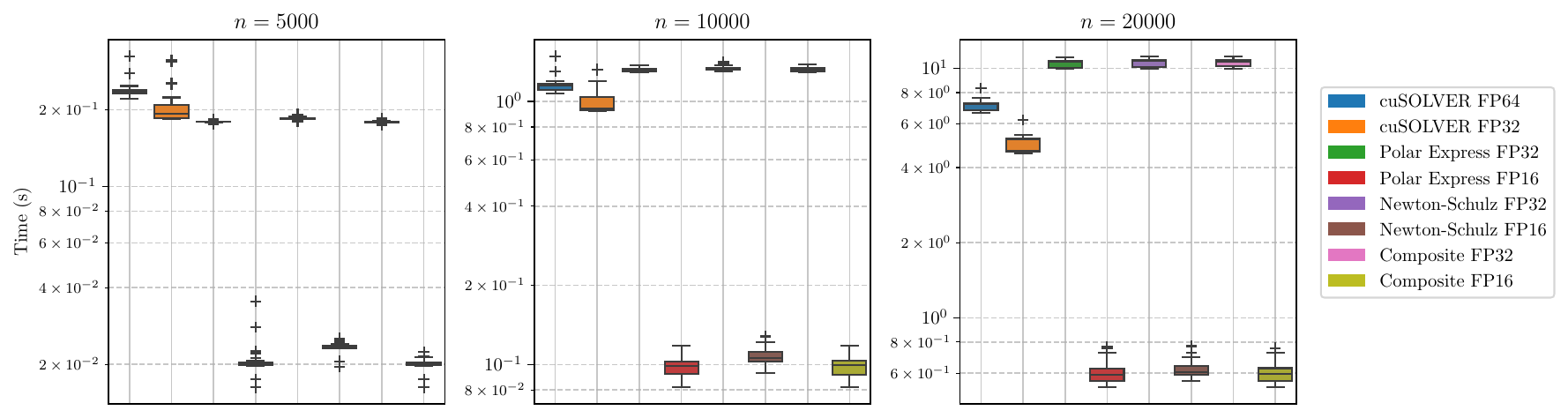}

    \caption{Boxplots for different PSD cone projection methods' \emph{execution time} on H100 GPU. \label{fig:psd_time_boxplot_H100}}
    
    
\end{figure}

\begin{figure}[ht]
    \centering
    \includegraphics[width=\textwidth]{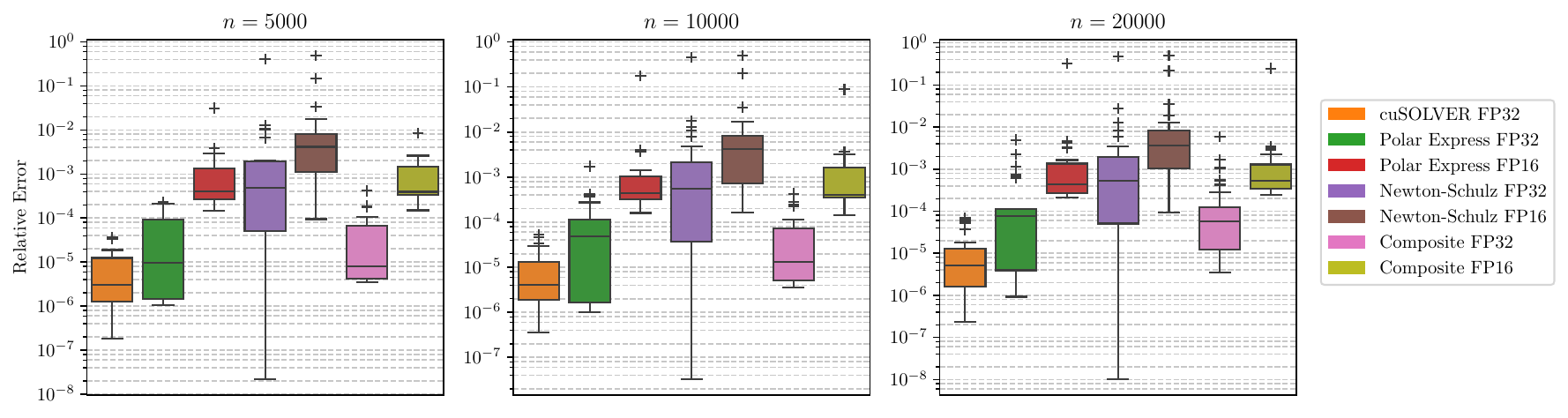}

    \caption{Boxplots for different PSD cone projection methods' \emph{relative error} on H100 GPU. \label{fig:psd_error_boxplot_H100}}

    
\end{figure}

%% file: figs/psd_projection_tables/H100_stats_5000.tex
\begin{table}[ht]
\begin{adjustbox}{width=\linewidth}
\begin{tabular}{lllllllll}
\toprule
{} & \multicolumn{3}{c}{Relative Error} & \multicolumn{3}{c}{Time (s)} \\
{} & {\quad Mean} & {\quad Median} & {\quad Std.} & {\quad Mean} & {\quad Median} & {\quad Std.} \\
\midrule
\CSdouble & \num{0.00e+00} & \num{0.00e+00} & \num{0.00e+00} & \num{2.38e-01} & \num{2.36e-01} & \num{1.88e-02} \\
\CSsingle & \num{7.96e-06} & \num{3.10e-06} & \num{1.06e-05} & \num{2.05e-01} & \num{1.93e-01} & \num{3.28e-02} \\
\PEsingle & \num{5.43e-05} & \num{9.70e-06} & \num{7.31e-05} & \num{1.79e-01} & \num{1.79e-01} & \num{1.85e-03} \\
\PEhalf & \num{1.83e-03} & \num{4.08e-04} & \num{5.38e-03} & \num{2.06e-02} & \num{2.01e-02} & \num{3.34e-03} \\
\NSsingle & \num{1.47e-02} & \num{4.77e-04} & \num{7.07e-02} & \num{1.85e-01} & \num{1.85e-01} & \num{2.42e-03} \\
\NShalf & \num{2.49e-02} & \num{4.14e-03} & \num{8.70e-02} & \num{2.32e-02} & \num{2.33e-02} & \num{1.38e-03} \\
\CPsingle{} & \num{4.93e-05} & \num{7.95e-06} & \num{8.21e-05} & \num{1.79e-01} & \num{1.79e-01} & \num{2.51e-03} \\
\CPhalf & \num{1.05e-03} & \num{3.95e-04} & \num{1.50e-03} & \num{1.99e-02} & \num{2.01e-02} & \num{1.39e-03} \\
\bottomrule
\end{tabular}
\end{adjustbox}
\caption{Statistics of the PSD projection methods on datasets of size 5000 for H100 GPU. \label{tab:benchmark_stats_5000_H100}}
\end{table}

%% file: figs/psd_projection_tables/H100_stats_10000.tex
\begin{table}[ht]
\begin{adjustbox}{width=\linewidth}
\begin{tabular}{lllllllll}
\toprule
{} & \multicolumn{3}{c}{Relative Error} & \multicolumn{3}{c}{Time (s)} \\
{} & {\quad Mean} & {\quad Median} & {\quad Std.} & {\quad Mean} & {\quad Median} & {\quad Std.} \\
\midrule
\CSdouble & \num{0.00e+00} & \num{0.00e+00} & \num{0.00e+00} & \num{1.16e+00} & \num{1.15e+00} & \num{7.47e-02} \\
\CSsingle & \num{1.08e-05} & \num{4.07e-06} & \num{1.53e-05} & \num{1.00e+00} & \num{9.43e-01} & \num{1.04e-01} \\
\PEsingle & \num{1.34e-04} & \num{4.91e-05} & \num{3.09e-04} & \num{1.32e+00} & \num{1.32e+00} & \num{2.43e-02} \\
\PEhalf & \num{6.37e-03} & \num{4.47e-04} & \num{3.08e-02} & \num{9.81e-02} & \num{9.87e-02} & \num{8.67e-03} \\
\NSsingle & \num{1.61e-02} & \num{5.54e-04} & \num{7.87e-02} & \num{1.34e+00} & \num{1.34e+00} & \num{2.70e-02} \\
\NShalf & \num{2.65e-02} & \num{4.23e-03} & \num{9.07e-02} & \num{1.08e-01} & \num{1.06e-01} & \num{8.49e-03} \\
\CPsingle{} & \num{6.20e-05} & \num{1.31e-05} & \num{9.72e-05} & \num{1.32e+00} & \num{1.32e+00} & \num{2.69e-02} \\
\CPhalf & \num{3.64e-03} & \num{4.00e-04} & \num{1.52e-02} & \num{9.82e-02} & \num{9.93e-02} & \num{8.90e-03} \\
\bottomrule
\end{tabular}
\end{adjustbox}
\caption{Statistics of the PSD projection methods on datasets of size 10000 for H100 GPU. \label{tab:benchmark_stats_10000_H100}}
\end{table}

%% file: figs/psd_projection_tables/H100_stats_20000.tex
\begin{table}[t]
\begin{adjustbox}{width=\linewidth}
\begin{tabular}{lllllllll}
\toprule
{} & \multicolumn{3}{c}{Relative Error} & \multicolumn{3}{c}{Time (s)} \\
{} & {\quad Mean} & {\quad Median} & {\quad Std.} & {\quad Mean} & {\quad Median} & {\quad Std.} \\
\midrule
\CSdouble & \num{0.00e+00} & \num{0.00e+00} & \num{0.00e+00} & \num{7.14e+00} & \num{7.20e+00} & \num{3.21e-01} \\
\CSsingle & \num{1.42e-05} & \num{5.21e-06} & \num{2.09e-05} & \num{4.91e+00} & \num{4.66e+00} & \num{3.73e-01} \\
\PEsingle & \num{4.01e-04} & \num{7.80e-05} & \num{9.50e-04} & \num{1.05e+01} & \num{1.06e+01} & \num{3.70e-01} \\
\PEhalf & \num{1.08e-02} & \num{4.45e-04} & \num{5.60e-02} & \num{6.03e-01} & \num{5.93e-01} & \num{6.06e-02} \\
\NSsingle & \num{1.68e-02} & \num{5.34e-04} & \num{8.28e-02} & \num{1.05e+01} & \num{1.07e+01} & \num{3.99e-01} \\
\NShalf & \num{2.68e-02} & \num{3.69e-03} & \num{9.24e-02} & \num{6.22e-01} & \num{6.07e-01} & \num{5.32e-02} \\
\CPsingle{} & \num{3.61e-04} & \num{5.82e-05} & \num{1.07e-03} & \num{1.05e+01} & \num{1.07e+01} & \num{3.84e-01} \\
\CPhalf & \num{8.49e-03} & \num{5.25e-04} & \num{4.30e-02} & \num{6.04e-01} & \num{5.94e-01} & \num{6.02e-02} \\
\bottomrule
\end{tabular}
\end{adjustbox}
\caption{Statistics of the PSD projection methods on datasets of size 20000 for H100 GPU. \label{tab:benchmark_stats_20000_H100}}
\end{table}